\documentclass[12pt,notitlepage]{amsart}
\usepackage{latexsym,amsfonts,amssymb,amsmath,amsthm}
\usepackage{graphicx}
\usepackage{color}
\usepackage{mathrsfs}
\usepackage{amsmath, amsthm, amsfonts, amssymb, cite}
\usepackage{wrapfig}
\usepackage{stackrel}
\usepackage{color}
\usepackage{float}
\usepackage{enumitem}
\usepackage{mathtools}
\usepackage{multicol}
\usepackage[normalem]{ulem}

\usepackage[inner=1.0in,outer=1.0in,bottom=1.0in, top=1.0in]{geometry}

\newcommand{\R}{\ensuremath{\mathbb{R}}}
\newcommand{\mz}{\ensuremath{\mathbb Z}}
\newcommand{\Z}{\ensuremath{\mathbb Z}}

\newcommand{\mr}{\ensuremath{\mathbb R}}

\newcommand{\mq}{\ensuremath{\mathbb Q}}
\newcommand{\Q}{\ensuremath{\mathbb Q}}

\newcommand{\C}{\ensuremath{\mathbb C}}

\newcommand{\F}{\ensuremath{\mathbb F}}
\newcommand{\A}{\ensuremath{\mathbb A}}

\newcommand{\mymod}{\ensuremath{\negthickspace \negmedspace \pmod}}
\newcommand{\shortmod}{\ensuremath{\negthickspace \negthickspace \negthickspace \pmod}}

\newcommand{\intR}{\int_{-\infty}^{\infty}}

\newcommand{\sumstar}{\sideset{}{^*}\sum}

\newcommand{\GL}{\ensuremath{\mathrm{GL}}}

\DeclareMathOperator{\sgn}{sgn}
\DeclareMathOperator{\Spec}{Spec}

\newcommand{\cF}{\ensuremath{\mathcal{F}}}
\newcommand{\cG}{\ensuremath{\mathcal{G}}}
\newcommand{\cH}{\ensuremath{\mathcal{H}}}
\newcommand{\cK}{\ensuremath{\mathcal{K}}}
\newcommand{\cL}{\ensuremath{\mathcal{L}}}

\newcommand{\Fbar}{\overline{\F}}

\newcommand{\Fr}{\mathrm{Fr}}

\newcommand{\real}{\mathop{\rm Re}}
\newcommand{\imag}{\mathop{\rm Im}}
\newcommand{\arcsinh}{\mathop{\rm arcsinh}}

\newcommand{\Gal}{\mathop{\rm Gal}}

\renewcommand{\bf}[1]{\mathbf{#1}}

\newcommand{\eps}{\varepsilon}

\newcommand{\psum}{\mathop{\sum\nolimits^+}}

\newcommand{\Tr}{{\rm Tr}}

\newtheorem{thm}{Theorem}[section]

\newtheorem{lem}[thm]{Lemma}

\theoremstyle{plain}		
	\newtheorem{mytheo}{Theorem} [section]
	
	\newtheorem{myprop}[mytheo]{Proposition}
	\newtheorem{mycoro}[mytheo]{Corollary}
     \newtheorem{mylemma}[mytheo]{Lemma}
	\newtheorem{mydefi}[mytheo]{Definition}
	\newtheorem{myconj}[mytheo]{Conjecture}

\theoremstyle{remark}

\numberwithin{equation}{section}
\numberwithin{figure}{section}

\begin{document}
\title{The Weyl bound for Dirichlet $L$-functions of cube-free conductor}
\author{Ian Petrow} 
\email{ian.petrow@math.ethz.ch}
\address{ETH Z\"urich\\
Department of Mathematics\\ 
R\"amistrasse 101\\
8092 Z\"urich\\
Switzerland}

\author{Matthew P. Young} 
\email{myoung@math.tamu.edu}
\address{Department of Mathematics\\
	  Texas A\&M University\\
	  College Station\\
	  TX 77843-3368\\
		U.S.A.}

 \thanks{The first author was supported by Swiss national science foundation grant PZ00P2\_168164. \\
 This material is based upon work supported by the National Science Foundation under agreement No. DMS-1702221 (M.Y.).  Any opinions, findings and conclusions or recommendations expressed in this material are those of the authors and do not necessarily reflect the views of the National Science Foundation.
 }

 \begin{abstract}
 We prove a Weyl-exponent subconvex bound for any Dirichlet $L$-function of cube-free conductor. We also show a bound of the same strength for certain $L$-functions of self-dual $\GL_2$ automorphic forms that arise as twists of forms of smaller conductor.
 \end{abstract}

 \maketitle

 \section{Introduction}
 Subconvex estimates for $L$-functions play a major role in modern analytic number theory.   
The first subconvex estimate is due to 
Weyl and Hardy-Littlewood, who showed
that \begin{equation}\label{WeylBound}\zeta(1/2+it) \ll_\eps (1+|t|)^{\frac{1}{6}+\eps}.\end{equation} 
The exponent $1/6$ appearing in \eqref{WeylBound} is a consequence of Weyl's differencing method for estimating exponential sums, introduced in 1916.  This method itself is important for studying equidistribution and has immediate applications to lattice point counting problems.

Today we call a subconvex bound of the form $L(1/2, \pi) \ll_{\eps} Q(\pi)^{1/6+\varepsilon}$ \emph{the Weyl bound}, where $Q(\pi)$ is the analytic conductor of the automorphic $L$-function $L(1/2, \pi)$.  The Weyl bound is only known in a few cases, notably for quadratic twists of certain self-dual $\GL_2$ automorphic forms; see \cite{CI} \cite{Ivic} \cite{YoungCubic} \cite{PetrowYoung} for example.

Estimating the Dirichlet $L$-functions $L(1/2,\chi)$ of conductor $q$ as $q \to \infty$ is analogous to estimating $\zeta(1/2+it)$ as $t \to \infty$, but the former is a harder and more arithmetic problem. In 1963, Burgess \cite{Burgess} 
showed by a completely different method that
\begin{equation}
\label{Burgess}L(1/2, \chi) \ll_\eps q^{\frac{3}{16}+\eps}.\end{equation} Burgess's method required new ideas, in particular it uses the Riemann Hypothesis for curves over finite fields. Note that the Burgess exponent of $3/16$ falls short of the exponent $1/6$ found by Weyl. Curiously, the exponent $3/16$ often re-occurs in the modern incarnations of these problems, see \cite{BlomerHarcosMichel} \cite{BlomerHarcos}  \cite{WuGL1timesGL2} \cite{WuGL1} for example. 

Even for the case of Dirichlet $L$-functions, the Burgess bound has only been improved in some limited special cases.  
In a breakthrough, Conrey and Iwaniec \cite{CI} 
obtained a Weyl-quality bound for \emph{quadratic} characters of odd conductor using techniques from automorphic forms and Deligne's solution of the Weil conjectures for varieties over finite fields.
Another class of results, such as \cite{BarbanLinnikTshudakov} and \cite{HeathBrownHybrid}, consider situations where the conductor $q$ of $\chi$ runs over prime  powers or otherwise has some special factorizations. 
Notably, Mili\'cevi\'c \cite{Milicevic} recently obtained a sub-Weyl subconvex bound when $q=p^n$ with $n$ large.

One of the main results of this paper  (see Corollary \ref{coro:WeylBound})  gives a Weyl-exponent subconvex bound for any Dirichlet $L$-function of cube-free conductor.  In particular, we give the first improvement on the Burgess bound for all Dirichlet $L$-functions of prime conductor.

\subsection{Statement of results}
\label{section:results}
 Let $q$ be a positive integer, and $\chi$ be a primitive Dirichlet character of conductor $q$.  
 Let $\mathcal{H}_{it_j}(m, \overline{\chi}^2)$  denote the set  (possibly empty) of Hecke-normalized Hecke-Maass newforms of level $m|q$, central character $\overline{\chi}^2$ and spectral parameter $t_j$.  For $f \in \mathcal{H}_{it_j}(m, \overline{\chi}^2)$, $f \otimes \chi$ is a self-dual newform of level $q^2$ and trivial central character. 
 \begin{mytheo}
\label{thm:mainthmMaassEisenstein}
 Let notation be as above.  Assume $q$ is cube-free and $\chi$ is not quadratic.
 Then for some $B > 2$ we have
\begin{equation}
 \sum_{m|q} \sum_{|t_j| \leq T} \sum_{f  \in \mathcal{H}_{it_j}(m, \overline{\chi}^2)} L(1/2, f  \otimes \chi)^3 
 + 
 \int_{-T}^{T} 
 |L(1/2 + it, \chi)|^6  dt
 \ll_{\varepsilon} T^{B} q^{1+\varepsilon}.
\end{equation}
\end{mytheo}
 
Theorem \ref{thm:mainthmMaassEisenstein} generalizes the celebrated result of Conrey and Iwaniec \cite{CI} which assumed $\chi$ is the quadratic character of odd, square-free conductor $q$.  The  central values appearing in Theorem \ref{thm:mainthmMaassEisenstein} are nonnegative \cite{Waldspurger}  \cite{Guo}, which is crucial for obtaining the Weyl-quality subconvex bound for these central values.

A potential defect of Theorem \ref{thm:mainthmMaassEisenstein} is that, although it is consistent with the Lindel\"{o}f hypothesis in the $q$-aspect, it is weak in the $T$-aspect.  However, if $T \ll q^{\varepsilon}$ then it is sharp.
As in the work of \cite{YoungCubic}, we can obtain a hybrid result for $T \gg q^{\varepsilon}$.
\begin{mytheo}
 \label{thm:mainthmHybridVersion}
Let conditions be as in Theorem \ref{thm:mainthmMaassEisenstein}, and suppose that  $T \gg q^{\delta}$ for some $\delta>0$.  Then
\begin{equation}
\label{eq:MaassEisensteinFormsBoundHybridVersion}
 \sum_{m|q} \sum_{T \leq t_j < T+1} \sum_{f  \in \mathcal{H}_{it_j}(m, \overline{\chi}^2)} L(1/2, f  \otimes \chi)^3 
 + 
 \int_{T}^{T+1} 
 |L(1/2 + it, \chi)|^6  dt
 \ll_{\varepsilon, \delta} T^{1+\varepsilon} q^{1+\varepsilon}.
\end{equation}
\end{mytheo}

As a consequence, we obtain a Weyl-quality subconvex bound for Dirichlet $L$-functions simultaneously in $q$- and $t$-aspects:
\begin{mycoro}
\label{coro:WeylBound}
 Suppose $\chi$ has cubefree conductor $q$.  Then
 \begin{equation}
 |L(1/2 + it, \chi)| \ll_{\varepsilon}  q^{1/6+\varepsilon} (1+|t|)^{1/6+\varepsilon}.
\end{equation}
\end{mycoro}

 \begin{mycoro}\label{coro:twistmin}
 Let $p$ be an odd prime, and suppose $F$ is a Hecke-Maass newform of level $p^2$, trivial central character, and spectral parameter $t_F$.  If $F$ is not twist-minimal, then
 \begin{equation}
  L(1/2, F) \ll_{\varepsilon} (p(1+|t_F|))^{1/3+\varepsilon}.
 \end{equation}
\end{mycoro}
Here the assumption that $F$ is not twist minimal means there exists a newform $f$ of level $m$ dividing $p$ and a primitive Dirichlet character $\chi$ of conductor $p$ so that $F = f \otimes \chi$.  The central character of $F$, which is trivial by assumption, equals $\chi^2$ times the central character of $f$.  
Hence $f \in \mathcal{H}_{i t_F}(m, \overline{\chi}^2)$, and so Theorem \ref{thm:mainthmMaassEisenstein} applies. 
Another observation is that for $F$ of level $p^2$ and trivial central character, the condition that $F$ is twist-minimal is equivalent to the assertion that the local representation of $\GL_2(\mq_{p})$ associated to $F$ is supercuspidal (see e.g.\ \cite[Table (4.20)]{Gelbart}). 

Theorems \ref{thm:mainthmMaassEisenstein} and \ref{thm:mainthmHybridVersion} (and hence Corollary \ref{coro:twistmin}) also carry over to holomorphic modular forms.  
Let $S_{\kappa}(q, \overline{\chi}^2)$ denote the space of cusp forms of level $q$, central character  $\overline{\chi}^2$, and even weight $\kappa \geq 2$.  
 Let $\mathcal{H}_{\kappa}(m,\overline{\chi}^2)$ denote the set of Hecke-normalized newforms of level $m|q$ and central character $\overline{\chi}^2$. 
\begin{mytheo}
\label{thm:mainthmHolomorphic}
 Let notation be as above, with $q$ cube-free.  Then
 \begin{equation}
 \label{eq:holomorphicFormsBoundkappaSmallVersion}
\sum_{m|q} \sum_{\kappa \leq T}  \sum_{f \in \mathcal{H}_{\kappa}(m,\overline{\chi}^2)} L(1/2, f \otimes \chi)^3 \ll_{\eps} T^{B} q^{1+\varepsilon},
 \end{equation}
 for some $B > 2$.  Moreover, if there exist $\delta>0$ such that 
 $T \gg q^{\delta}$, then we have
 \begin{equation}
 \label{eq:holomorphicFormsBoundkappaHybridVersion}
\sum_{m|q} \sum_{T \leq \kappa < T + 1}  \sum_{f \in \mathcal{H}_{\kappa}(m,\overline{\chi}^2)} L(1/2, f \otimes \chi)^3 \ll_{\eps, \delta} T^{1+\varepsilon} q^{1+\varepsilon}.
 \end{equation}
 \end{mytheo}
The sum over $\kappa$ in \eqref{eq:holomorphicFormsBoundkappaHybridVersion} has at most one non-zero term, and is often empty.  Nonetheless, we include it so that \eqref{eq:holomorphicFormsBoundkappaHybridVersion}  aligns with the form of \eqref{eq:MaassEisensteinFormsBoundHybridVersion}.

\subsection{Remarks}
The reader may wonder why $q$ is restricted to be cube-free in the above results (coincidentally, the Burgess bound for character sums is stronger in certain ranges in case the conductor is cube-free, e.g. see \cite[Thm.\ 12.6]{IK}).  To explain this restriction on $q$, we need to outline the proof of Theorem \ref{thm:mainthmMaassEisenstein}.  As in the work of Conrey and Iwaniec \cite{CI}, we apply some standard tools: approximate functional equations, the Petersson/Kuznetsov formula, and Poisson summation.  The dual sum after Poisson summation in large part boils down to a certain character sum defined by
\begin{equation}
\label{eq:gdef}
 g(\chi,\psi) = \sum_{t,u \shortmod{q}} \chi(t) \overline{\chi}(t+1) \overline{\chi}(u) \chi(u+1) \psi(ut-1),
\end{equation}
where $\psi$ is a Dirichlet character modulo $q$.  
After the above steps, the problem essentially reduces to bounding
\begin{equation}
\label{eq:sketchdualmoment}
 \sum_{\psi \shortmod{q}} | L(1/2, \psi)|^4  g(\chi, \psi).
\end{equation}
Since the fourth moment of Dirichlet $L$-functions is of size $O_{\eps}(q^{1+\varepsilon})$, the sum \eqref{eq:sketchdualmoment} can be bounded by $O_{\eps}(q^{1+\varepsilon})$ times the maximum value of $|g(\chi,\psi)|$ as $\psi$ varies.  Here, the Riemann hypothesis of Deligne \cite{DeW2} plays a crucial role in proving $|g(\chi,\psi)| \ll_{\eps} q^{1+\varepsilon}$ for $q$ prime (see Section \ref{section:gboundp}), which then extends to square-free $q$ by multiplicativity.  In case $q=p^2$, we establish $|g(\chi,\psi)| \ll_{\eps} q^{1+\varepsilon}$ by elementary means (see Section \ref{section:gboundpsquared}), and hence this bound on $g(\chi,\psi)$ holds for cube-free $q$.  However, for $q=p^3$, it is no longer true that $|g(\chi,\psi)| \ll_{\eps} q^{1+\varepsilon}$ for all primitive $\psi$.  Rather, there exist many characters of conductor $p^3$ so that $|g(\chi,\psi)| \gg q p^{1/2}$.  
Barring an improved estimate for the sub-sum of \eqref{eq:sketchdualmoment} coming from these ``bad'' characters $\psi$, this extra factor of $p^{1/2}$ would propagate through all the estimates, and hence would presumably lead to (at best) the bound
\begin{equation}\label{eq:L12bad}
 |L(1/2+it, \chi)|^6 \ll_{\eps} q^{1+\varepsilon} p^{1/2} \qquad (q=p^3).
\end{equation}
This would imply $|L(1/2+it,\chi)| \ll_{\eps} q^{\frac{7}{36} + \varepsilon}$, and note $\frac{7}{36} > \frac{3}{16}$, so this would not improve on the Burgess bound. 

The analysis of $g(\chi,\psi)$ becomes more complicated for $q=p^n$ with larger $n$.  Since there are complementary methods well-suited to treat the depth-aspect (as in \cite{Milicevic} \cite{BlomerMili}, and  other papers), we content ourselves here with the restriction to $q$ cube-free.

Remark added August 28, 2019:
In \cite{PetrowYoungFourth}, written after the first version of the present paper, we have extended all the cubic moment bounds stated in Section \ref{section:results} to hold for arbitrary $q$.  More precisely, \cite{PetrowYoungFourth} contains proofs of Conjectures \ref{conj:characterSumBoundIntermediateCase} and \ref{conj:Zproperties} from the present paper, which are shown here to imply the cubic moment bounds for general $q$.

\subsection{Organization of the paper}

For the rest of the paper, we will focus almost entirely on the proof of Theorem \ref{thm:mainthmMaassEisenstein}.  The proof of Theorem \ref{thm:mainthmHybridVersion} follows the same approach, and the only change is in the behavior of the  weight function on the spectrum.  These archimedean aspects were already developed in \cite{YoungCubic}, so we can largely quote those results.  For brevity, we sketch the proof in Section \ref{section:Epilogue}.

The analogous results on the holomorphic forms (Theorem \ref{thm:mainthmHolomorphic}) are also similar to the Maass form cases, so we briefly sketch the necessary changes in Section \ref{section:Epilogue}.

\subsection{Convention}
\label{section:conventions}
The notation $A\ll B$ for quantities $A$ and $B$ means that there exists a constant $K$ such that $|A| \leq K B$ for all relevant $A$ and $B$, the value of which in each instance should be clear from context. If $p_1,\ldots,p_n$ are parameters, then $\ll_{p_1,\ldots, p_n}$ indicates that the constant $K$ may depend on $p_1,\ldots,p_n$. Implied constants also depend on the choices of implied constants already established in the proof, but we suppress this from this notation.  For example, if $|A| \leq K_1 B$ and $|B| \leq K_2 C$, then $A \ll C$ with $K_3 = K_1K_2$. A major purpose of this notation is to avoid excessive labelling of implied constants. The appearance of the parameter $\eps$ among the $p_i$ plays a similar role: each of these $\eps$ represents a quantity $\eps_{j}$ that may be taken to be arbitrarily small, and which may depend on all previous $\eps_1, \ldots, \eps_{j-1}$ appearing in the proof.

\subsection{Acknowledgements}
We would like to thank Emmanuel Kowalski for explaining his work on $\ell$-adic trace functions to us, which plays a crucial role in Section \ref{section:gboundp} of this paper. We also thank Philippe Michel for pointing out an oversight in an earlier version of that section of the paper, and for proposing a solution to it. Part of this work was accomplished during our visit to the Hausdorff Center in Bonn for the summer school on $L$-functions in 2018. We thank the Center for its support. Lastly, we thank the referees for their careful and thorough reading of this paper.

\section{Automorphic forms and $L$-functions}
\subsection{Cusp forms}
Let $q$ be a positive integer, and $\psi$ a Dirichlet character modulo $q$. For $t_j \in \R \cup i[-1/2,1/2]$ let $S_{it_j}(q,\psi)$ be the space of Maass cusp forms of level $q$, central character $\psi$, and spectral parameter $t_j$. 
Similarly, for $\kappa \geq 2$ we let $S_\kappa(q,\psi)$ be the space of holomorphic cusp forms of weight $\kappa$. Any $f \in S_{it_j}(q, \psi)$ admits a Fourier expansion 
\begin{equation}
\label{eq:FourierExpansionMaass}
 f(z) = 2 \sqrt{y} \sum_{n \neq 0} \lambda_{f}(n) e(nx) K_{it_j}(2 \pi |n| y),
\end{equation}
and similarly, if $f \in S_\kappa(q,\psi)$ we may write 
\begin{equation}
\label{eq:FourierExpansionHolomorphic}
 f(z) = \sum_{n=1}^{\infty} \lambda_f(n) n^{\frac{\kappa-1}{2}} e(nz).
\end{equation}
Now let $\cH_{it_j}(m,\psi)$ be the set of Hecke-Maass \emph{newforms} of level $m|q$, normalized so that $\lambda_f(1)=1$, and define similarly $\cH_\kappa(m,\psi)$. Recall the Petersson inner product on $S_{it_j}(q,\psi)$ or $S_\kappa(q,\psi)$ defined by 
\begin{equation*}
 \langle f, g \rangle_q := \int_{\Gamma_0(q) \backslash \mathbb{H}} y^{\kappa} f(z) \overline{g}(z) \frac{dx dy}{y^2},
\end{equation*}
where in the former case we take $\kappa =0$. With this normalization of the inner product, we have for any $f \in \cH_{it_j}(m,\psi)$ or $\cH_\kappa(m,\psi)$ by Rankin-Selberg theory and work of Iwaniec and Hoffstein-Lockhart \cite{IwaniecSmallEigenvalues,HoffsteinLockhart} that 
\begin{equation}
\label{eq:PeterssonNormMaassCuspForm}
  \langle f, f \rangle_q = \frac{q}{\cosh(\pi t_j)} (q(1+|t_j|))^{o(1)},
 \quad
 \text{or}
 \quad
 \langle f, f \rangle_q = \frac{q \Gamma(\kappa)}{ (4 \pi)^{\kappa-1}} (q \kappa )^{o(1)}.
\end{equation} 
In fact, we only use the upper bounds implicit in \eqref{eq:PeterssonNormMaassCuspForm}, which are due to Iwaniec. 

Any newform $f \in \mathcal{H}_*(m, \psi)$ satisfies the Hecke relation 
\begin{equation}
 \label{eq:HeckeRelation}
 \lambda_f(n_1) \lambda_f(n_2) = \sum_{d|(n_1,n_2)} \lambda_f(n_1 n_2/d^2) \psi(d).
\end{equation}

Recall that a Hecke-Maass newform $f$ is called even if $\lambda_f(-1) = 1$, and odd if $\lambda_f(-1) = -1$.  It is easy to see that the parity of $f \otimes \chi$ is the parity of $f$ times the parity of $\chi$.

By Atkin-Lehner-Li theory \cite{AtkinLehner, AL78} we have the following direct sum decomposition:
\begin{equation}
\label{eq:orthogonaldecompositionMaass}
 S_{it_j}(q, \psi) = \bigoplus_{\ell m=q} \bigoplus_{f \in \mathcal{H}_{it_j}(m, \psi)} S_{it_j}(\ell,f,\psi),
\end{equation}
where $S_{it_j}(\ell,f,\psi) = \text{span} \{ f(d z): d | \ell \}$, and similarly for holomorphic forms, where each instance of $it_j$ is replaced by $\kappa$. The direct sums in \eqref{eq:orthogonaldecompositionMaass} are orthogonal with respect to the Petersson inner product. 

For any $f \in \mathcal{H}_{it_j}(m,\overline{\chi}^2)$ with $m|q$, we have by \cite[Prop.\ 3.8(iii)]{JacquetLanglands} that $f \otimes \chi \in \cH_{it_j}(q^2,1)$, and similarly for holomorphic forms. See also \cite[Thm.\ 3.1(ii)]{AL78} for a classical proof of this fact.

\subsection{Eisenstein series}
\label{section:EisensteinSeriesDefinitions}
Let \begin{equation}
\label{eq:EisensteinDefinition}
 E_{\chi_1, \chi_2}(z,1/2+it) = e_{\chi_1, \chi_2}(y, 1/2+it) + 
 2 \sqrt{y} \sum_{n \neq 0} \lambda_{\chi_1, \chi_2, t}(n) e(nx) K_{it }(2 \pi |n| y),
\end{equation}
where $\chi_1, \chi_2$ are primitive Dirichlet characters modulo $q_1, q_2$, respectively,  
\begin{equation*}
\lambda_E(n)
= 
 \lambda_{\chi_1, \chi_2, t}(n)
 =
 \chi_2(\sgn(n)) \sum_{ab=|n|} \chi_1(a) \overline{\chi_2}(b) a^{-it} b^{it},
\end{equation*}
and $e_{\chi_1, \chi_2}(y,s) = c y^{s} + c' y^{1-s}$, for certain constants $c,c'$. Note that the definition \eqref{eq:EisensteinDefinition} corresponds to the ``completed'' Eisenstein series $E_{\chi_1, \chi_2}^*(z, 1/2+it)$ in \cite{YoungEisenstein}, so some care is needed when we quote results from that reference. 
Then $E_{\chi_1,\chi_2}$ is of level $m=q_1q_2$ and central character $\chi_1 \overline{\chi_2}$, and is an eigenfunction of all the Hecke operators, and so \eqref{eq:HeckeRelation} also holds for $\lambda_E(n)$. These are, by definition, the newform Eisenstein series. For two arbitrary Dirichlet characters $\chi$ and $\psi$, let us write $\chi \simeq \psi$ if the underlying primitive characters of $\chi$ and $\psi$ are equal.
With this notation, we denote the set of newform Eisenstein series by $$\cH_{it, \text{Eis}}(m,\psi) = \{E_{\chi_1, \chi_2}(z, 1/2+it) : q_1q_2=m \text{ and } \chi_1\overline{\chi_2} \simeq \psi\}.$$ 
In particular, if $E \in \cH_{it,\text{Eis}}(m,\psi)$, then $\lambda_E(1)=1$ and the Hecke relations hold for $\lambda_E(n)$ exactly as they do for $\lambda_f(n)$. 

The space $\mathcal{E}_{it}(q,\psi)$, for $t \neq 0$, admits a formal inner product $\langle \cdot, \cdot \rangle_{\text{Eis}}$ induced by $$\tfrac{1}{4 \pi} \langle E_{\mathfrak{a}}(z,1/2+it, \psi), E_{\mathfrak{b}}(z, 1/2+it, \psi) \rangle_{\text{Eis}} = \delta_{\mathfrak{a} = \mathfrak{b}}.$$  With this definition of the inner product, we have in perfect analogy to \eqref{eq:PeterssonNormMaassCuspForm} that 
\begin{equation}
\label{eq:Echi1chi2InnerProduct}
 \langle E_{\chi_1, \chi_2}(z, 1/2+it), E_{\chi_1, \chi_2}(z, 1/2+it) \rangle_{\text{Eis}} = 
 \frac{q^{1+o(1)}}{\cosh(\pi t )} |L(1+2it, \chi_1 \chi_2)|^2 .
\end{equation}
This equation can be deduced from \cite[(8.13), (8.10)]{YoungEisenstein}, keeping in mind the normalization of the completed Eisenstein series (see \cite[\S 4]{YoungEisenstein}).

There exists an Atkin-Lehner-Li theory for the space $\mathcal{E}_{it}(q,\psi)$, for $t \neq 0$, and a decomposition into spaces of old forms completely analogous to \eqref{eq:orthogonaldecompositionMaass}. This decomposition is orthogonal with respect to $\langle \cdot, \cdot \rangle_{\text{Eis}}$, and is explained thoroughly in \cite[\S 8]{YoungEisenstein}.

Lastly, we define, for $\chi_1 \overline{\chi_2} \simeq \overline{\chi}^2$ with $\chi$ primitive of conductor $q$,
\begin{equation}
 \label{eq:LfunctionEisensteinSeriesTwistFormula}
 L(s, E_{\chi_1, \chi_2, t} \otimes \chi)
 = \sum_{n=1}^{\infty} \frac{\lambda_{\chi_1, \chi_2, t}(n) \chi(n)}{n^s}
= L(s+it, \chi \chi_1) L(s-it,  \chi \overline{\chi_{2}}).
\end{equation}
We claim that \eqref{eq:LfunctionEisensteinSeriesTwistFormula} defines the true automorphic $L$-function of conductor $q^2$. To see this, check that locally all the solutions to $\chi_1 \overline{\chi_2} \simeq \overline{\chi}^2$ with $q_1 q_2 | q$ arise from $\chi_1 = 1, \chi_2 = \chi^2$ or $\chi_2 = 1$, $\chi_1 = \overline{\chi}^2$.  Hence both $\chi \chi_1$ and $\chi \overline{\chi_2}$ are primitive of conductor $q$.

\subsection{Root Numbers}\label{rootnumber} Although the theorems in this paper do not depend on the precise values of the root numbers of the forms $f \otimes \chi$, formulas for these are useful when interpreting the main results of this paper. If $\chi$ is primitive modulo $q$, $m \mid q$ and $f \in \mathcal{H}_{it_j}(m,\overline{\chi}^2)$ or $f \in  \mathcal{H}_{it, {\rm Eis}}(m,\overline{\chi}^2)$, then the root number $\epsilon(f \otimes \chi)$ is equal to the parity of $f$. If $f \in \mathcal{H}_{\kappa}(m,\overline{\chi}^2)$, then $\epsilon(f \otimes \chi) = i^{-\kappa} \chi(-1)$.
These formulas follow from local computations at finite primes using \cite[Prop.\ 3.8(iii)]{JacquetLanglands} and the explicit formulas for root numbers at the archimedean place found just above \cite[Thm.\ 5.15]{JacquetLanglands}. See also \cite[\S1]{LiEpsFactorsAndnCloseness1}.

\subsection{Bruggeman-Kuznetsov}  
Let $B_{it_j}(q,\psi)$ denote an orthogonal basis for $S_{it_j}(q,\psi)$, and $B_{it,\text{Eis}}(q,\psi)$ denote an orthogonal basis for $\mathcal{E}_{it}(q,\psi)$ when $t \neq 0$.  Let $h(t)$ be a function holomorphic in the strip $|\text{Im}(t)| \leq \frac12 + \delta$, satisfying $h(t) = h(-t)$, and $|h(t)| \ll (1+|t|)^{-2-\delta}$ for some $\delta > 0$. Recall the twisted Kloosterman sum \begin{equation*}
 S_{\psi}(m,n;c) = \sumstar_{y \shortmod{c}} \overline{\psi}(y) e\Big(\frac{my + n \overline{y}}{c}\Big)
,
\end{equation*}
where the $*$ on the sum indicates that $(y,c) = 1$,
and let $c_t = \frac{4\pi}{\cosh(\pi t)}$. Then, for $mn > 0$ we have (see e.g. \cite[(10.2)]{YoungEisenstein}) 
\begin{multline*}
 \sum_{ t_j} h(t_j) c_{t_j} \sum_{f \in B_{it_j}(q, \psi)} \frac{\lambda_f(m) \overline{\lambda_f(n)}}{\langle f, f \rangle_q} 
 + 
 \frac{1}{4 \pi} \intR h(t) c_t
 \sum_{E \in B_{it, \text{Eis}}(q,\psi)} 
 \frac{\lambda_E(m) \overline{\lambda_E(n)}}{\langle E, E \rangle_{\text{Eis}}} \, dt 
 \\
 =
\delta_{m=n} g_0 + \sum_{c \equiv 0 \shortmod{q}} \frac{S_{\psi}(m,n;c)}{c} g^{+}\Big(\frac{4 \pi \sqrt{mn}}{c}\Big),
\end{multline*}
where
\begin{equation}
\label{eq:gplusdef}
g_0 = \frac{1}{\pi} \intR t \tanh(\pi t) h(t)\, dt,
\quad \text{ and } \quad
g^{+}(x) = 2i \intR \frac{J_{2it}(x)}{\cosh(\pi t)} t h(t)\, dt.
\end{equation}

It was shown by the first author \cite[\S7]{PetrowPetersson} that there exists certain positive weights $\rho_f(\ell) = \ell^{o(1)}$ such that if $(n_1 n_2,q) = 1$, then
\begin{equation} 
\label{eq:KuznetsovSieveFormulaCuspForms}
 \sum_{\ell m=q} \sum_{f \in \mathcal{H}_{it_j}(m, \psi)} \frac{\lambda_f(n_1) \overline{\lambda_f(n_2)}}{\langle f, f \rangle_q} \frac{1}{\rho_f(\ell)}
 =\sum_{f \in B_{it_j}(q, \psi)} \frac{\lambda_f(n_1) \overline{\lambda_f(n_2)} }{\langle f, f \rangle_q}.
 \end{equation}
 The weight $\rho_f(\ell)$ is a certain explicit function of the Hecke eigenvalues of $f$. It was shown by the second author \cite[\S 8.5 and Lem.\ 8.3]{YoungEisenstein} that an analogous formula holds for the Eisenstein series, namely
\begin{equation} 
\label{eq:KuznetsovSieveFormulaEisenstein}
 \sum_{\ell m=q} \sum_{E \in \mathcal{H}_{it, \text{Eis}}(m, \psi)} \frac{\lambda_E(n_1) \overline{\lambda_E(n_2)}}{\langle E, E \rangle_{\text{Eis}}} \frac{1}{\rho_E(\ell)}
 =\sum_{E \in B_{it, \text{Eis}}(q, \psi)} \frac{\lambda_E(n_1) \overline{\lambda_E(n_2)} }{\langle E, E \rangle_{\text{Eis}}},
 \end{equation}
where $\rho_E(\ell)$ is given by the same function of the Hecke eigenvalues of $E$ as $\rho_f(\ell)$.

 Let 
$$
w_{f,\ell} = c_{ t_j} \frac{1}{\langle f, f \rangle_q} \frac{1}{\rho_f(\ell)} \quad \text{ and } \quad w_{E,\ell} =  c_t \frac{1}{\langle E, E \rangle_{\text{Eis}}} \frac{1}{\rho_E(\ell)}$$
for $f \in \cH_{it_j}(m , \psi)$ and $E \in \cH_{it, \text{Eis}}(m , \psi)$. 
Note that $c_{t_j}>0$ for any $f \in \cH_{it_j}(m , \psi)$, including any exceptional cases where $t_j \in i \R$. More precisely, we have by \eqref{eq:PeterssonNormMaassCuspForm} (see also \cite[Sections 8.4, 8.5]{YoungEisenstein} for the Eisenstein case)
\begin{equation}
\label{eq:ctjweightformula}
w_{f,\ell} = q^{-1} (q(1+|t_j|))^{o(1)} \quad \text{ and }\quad w_{E,\ell} = \frac{q^{-1}  (q(1+|t|))^{o(1)}}{|L(1+2it, \chi_1 \chi_2)|^2}.
\end{equation}
Note that if $\chi_1 \chi_2$ is the trivial character, then this weight vanishes to order $2$ at $t=0$, which is the situation encountered in \cite{CI}. Indeed, there $q$ is square-free and $\chi$ is quadratic, hence the only solution to $\chi_1 \overline{\chi_2} \simeq \overline{\chi}^2$ with $q_1 q_2 | q$ is $q_1 = q_2 = 1$, $\chi_1 = \chi_2 = 1$.  
By the hypothesis in Theorem \ref{thm:mainthmMaassEisenstein} that $\chi$ is not quadratic, we have $\chi_1\chi_2$ is not trivial (see the discussion following \eqref{eq:LfunctionEisensteinSeriesTwistFormula}), and hence $w_{E,\ell} \gg q^{-1} (q(1+|t|))^{-\varepsilon}$ for all $t \in \mr$. This is the only place where the hypothesis that $\chi$ is not quadratic is used in this paper, which is for convenience of notation only.

In summary, we have established the following. 
\begin{myprop}\label{prop:BruggemanKuznetsovFormula}
Suppose $\chi$ is primitive of conductor $q$, and not quadratic.
There exist  positive weights $w_{f,\ell} \gg q^{-1} (q(1+|t_j|))^{-\varepsilon}$, and $w_{E, \ell} \gg q^{-1} (q(1+|t|))^{-\varepsilon}$ so that for any $(n_1 n_2,q)=1$ and $n_1 n_2 > 0$ we have 
\begin{multline}
\label{eq:Kuznetsov}
\sum_{ t_j} h(t_j) 
\sum_{\ell m=q} \sum_{f \in \mathcal{H}_{it_j}(m,\overline{\chi}^2)} w_{f,\ell} \lambda_{f}(n_1) \overline{\lambda_{f}(n_2) }
\\ 
 +
\frac{1}{4 \pi} \intR h(t)  \sum_{\ell m=q} \sum_{E \in \mathcal{H}_{it,\text{Eis}}(m, \overline{\chi}^2)} w_{E, \ell} \lambda_{E}(n_1) \overline{\lambda_{E}(n_2)}  dt
\\
=
\delta_{n_1=n_2} g_0 + \sum_{c \equiv 0 \shortmod{q}} \frac{S_{\overline{\chi}^2}(n_1,n_2;c)}{c} g^{+}\Big(\frac{4 \pi \sqrt{n_1 n_2}}{c}\Big).
\end{multline}
\end{myprop}

We also need the opposite-sign case of Proposition \ref{prop:BruggemanKuznetsovFormula}, i.e., when $n_1 n_2 < 0$.  The formula is identical to \eqref{eq:Kuznetsov} except that $g^{+}(x)$ is replaced by $g^{-}(x)$ defined by
\begin{equation}
\label{eq:gminusdef}
 g^{-}(x) =  \frac{4}{\pi} \int_0^{\infty} K_{2it}(x) \sinh(\pi t) t h(t)\, dt.
\end{equation}

\section{Conventions and terminology for weight functions}
\label{section:definitionsandconventions}

We begin with a useful definition from \cite{KiralPetrowYoung}.  
Let $\cF$ be an index set and $X=X_T: \cF \to \R_{\geq 1}$ be a function of $T \in \cF$. 
\begin{mydefi}\label{inert}
A family $\{w_T\}_{T\in \cF}$ of smooth  
functions supported on a product of dyadic intervals in $\R_{>0}^d$ is called \emph{$X$-inert} if for each ${\bf a} \in \Z_{\geq 0}^d$ we have $$C({\bf a}):= \sup_{T \in \cF} \sup_{{\bf t} \in \R_{>0}^d} X_T^{-{\bf a}.{\bf 1}}\left| {\bf t}^{\bf a} w_T^{({\bf a})}({\bf t}) \right| < \infty.$$
\end{mydefi}

It is also convenient for later purposes to slightly generalize the above notion of a family of $X$-inert functions.  
\begin{mydefi}
Suppose that $W_T(x, {\bf t})$ with $T \in \mathcal{F}$  is a family of smooth functions, where ${\bf t} \in \mr^d$.  
We say that $\{W_T\}_{T \in \mathcal{F}}$ forms an $X$-inert family with respect to ${\bf t}$ if $W$ has dyadic support in terms of ${\bf t}$ and if for each ${\bf a}$, $k$ and $x$ we have 
\begin{equation*}
C_k(x,{\bf a}) :=  \sup_{T \in \cF} \sup_{{\bf t} \in \R_{>0}^d}  X_T^{-{\bf a}.{\bf 1}} \Big| {\bf t}^{\bf a} \frac{\partial^{{\bf a}}}{\partial {\bf t}^{\bf a}} \frac{\partial^k}{\partial x^k} W_T(x, {\bf t}) \Big| < \infty.
\end{equation*}  
\end{mydefi}

As a convention, we may write $w(x,\cdot)$ as shorthand to represent $w(x,{\bf t})$.  We may then state that $w(x,\cdot)$ is $X$-inert with respect to ${\bf t}$, which allows us to concisely track the behavior of $w$ with respect to the suppressed variables.

\section{Setting up the moment problem}
For $T\geq 1$, let \begin{equation}
\label{eq:h0def}
 h_0(t) = \exp(-(t/T)^2) \frac{(t^2+\frac14)}{T^2}.
\end{equation}
Note $h_0(t) > 0$ for $t \in \mr$ as well as $-\frac12 < it < \frac12$.  Moreover, $h_0(t) \gg T^{-2}$ for $t \ll T$.

In this paper we are concerned with estimating the following moment of $L$-functions:
\begin{multline}\label{Mpmqchidef}
\mathcal{M} (q,\chi)
:= 
\sum_{t_j} h_0(t_j) \sum_{\ell m=q} \psum_{f \in \mathcal{H}_{it_j}(m,\overline{\chi}^2)} w_{f,\ell} L(1/2, f \otimes \chi)^3
\\
+
\frac{1}{4 \pi} \intR h_0(t)  \sum_{\ell m=q} \psum_{E \in \mathcal{H}_{it,\text{Eis}}(m, \overline{\chi}^2)} w_{E, L} L(1/2, E \otimes \chi)^3 dt,
\end{multline}
where the $+$ over the sums represents Maass forms or Eisenstein series with even parity.
\begin{mytheo}\label{helperthm1}
If $\chi$ has cube-free conductor and is not quadratic, then we have $$\mathcal{M} (q,\chi) \ll_\eps T^Bq^\eps.$$
\end{mytheo}
Theorem \ref{helperthm1} implies Theorem \ref{thm:mainthmMaassEisenstein}. Indeed, let $\chi_1 = 1$ and $\chi_2$ be the primitive character underlying $\chi^2$. Then $E=E_{\chi_1,\chi_2,t}$ occurs in $\mathcal{H}_{it,\text{Eis}}(m, \overline{\chi}^2)$ for some $m|q$, and we have for this $E$ that 
\begin{equation*}
L(1/2, E \otimes \chi) = |L(1/2+it, \chi)|^2.
\end{equation*}
We have as well that $L(1/2, f \otimes \chi)\geq 0$ by \cite{Guo} (see also \eqref{eq:LfunctionEisensteinSeriesTwistFormula} for the nonnegativity in the Eisenstein case), so that Theorem \ref{thm:mainthmMaassEisenstein} follows from Theorem \ref{helperthm1} by \eqref{eq:ctjweightformula}.

\subsection{Approximate functional equation}
For $j=1,2$, let 
\begin{equation}
\label{eq:Vjdef}
 V_{j}(y,t) = \frac{1}{2 \pi i} \int_{(\sigma)} y^{-s} \frac{\Gamma_{\mr}(1/2+\delta+s+it)^j \Gamma_{\mr}(1/2+\delta+s-it)^j}{\Gamma_{\mr}(1/2+\delta+it)^j \Gamma_{\mr}(1/2+\delta-it)^j} \frac{G_j(s)}{s} ds,
\end{equation}
where $\Gamma_{\mr}(s) = \pi^{-s/2} \Gamma(s/2)$, $\delta \in \{0, 1\}$, and $\sigma$ is to the right of all poles of the integrand.  We take $G_1(s) = e^{2s^2}$ and $G_2(s) = e^{4s^2}$. Here $V_j(x, t)$ is a smooth function on $x>0$ with rapid decay for $x \gg 1+|t|^j$. See Section \ref{section:WeightFunctions1} for more precise estimates for $V_j$.  
\begin{mylemma}\label{afe}
Suppose $m \mid q$ and $f\in \mathcal{H}_{it_j}(m, \overline{\chi}^2)$ is even.  
We have \begin{equation*}
 L(1/2, f \otimes \chi)^3
 =   
\sum_{(d,q) = 1} \frac{4}{d} 
 \sum_{n_1, n_2, n_3 } \frac{\lambda_f(n_1)\overline{\lambda_f}(n_2 n_3)  \chi(n_1) \overline{\chi}(n_2 n_3)}{\sqrt{n_1 n_2 n_3}}  
 V_1\Big(\frac{n_1}{q}, t_j\Big)
 V_2\Big(\frac{n_2 n_3 d^2}{q^2}, t_j \Big),
\end{equation*}
and similarly for $L(1/2, E \otimes \chi)^3$ for $m \mid q$ and $E \in \mathcal{H}_{it, {\rm Eis}}(m, \overline{\chi}^2)$ even. The parity parameter $\delta$ implicit in the definition of $V_j$ is equal to the parity of $\chi$. 
\end{mylemma}
\begin{proof}
Since $f$ is even, the root number $\epsilon(f \otimes \chi)$ is $+ 1$.  For $f$ a Maass newform of spectral parameter $t_j$, a standard approximate functional equation \cite[Theorem 5.3]{IK}  gives 
\begin{equation}
\label{eq:AFE}
 L(1/2, f \otimes \chi)
 = 2
 \sum_{n_1} \frac{\lambda_f(n_1) \chi(n_1)}{\sqrt{n_1}}  V_1\Big(\frac{n_1}{q}, t_j \Big),
\end{equation}
where $\delta = 0$ if $\chi$ is even and $\delta = 1$ if $\chi$ is odd.
Similarly we have
\begin{equation}
\label{eq:AFEsquared}
 L(1/2, f \otimes \chi)^2 = 
 2 \sum_{(d,q) = 1} \frac{1}{d}
 \sum_{n_2, n_3} \frac{\overline{\lambda_f}(n_2 n_3) \overline{\chi}(n_2 n_3)}{\sqrt{n_2 n_3}} V_2\Big(\frac{n_2 n_3 d^2}{q^2}, t_j \Big),
\end{equation}
where the 
conjugates appear for convenience since $\lambda_f(n) \chi(n) \in \mr$, and the
sum over $d$ arises from
the Hecke relation \eqref{eq:HeckeRelation}.

The product of \eqref{eq:AFE} and \eqref{eq:AFEsquared} gives the formula in the statement of the lemma. 
\end{proof}

\subsection{Bruggeman-Kuznetsov}
\label{section:BruggemanKuznetsovApplication}
 Let $N_1,N_2,N_3,C \gg 1$, and let $w_0 (\cdot) =w_0(n_1,n_2,n_3,c)$ be a family of $1$-inert functions (depending on $q,T,N_j,C$) with dyadic support on $n_j \asymp N_j$ and $c \asymp C$. 
 Let $J^\pm_0 = J^\pm_0(x, n_1,n_2,n_3,c)$ be defined by 
\begin{equation}
\label{eq:Jpm0def}
J^+_0(x, n_1,n_2,n_3,c) =  w_0(\cdot) \intR \frac{J_{2it}(x)}{\cosh( \pi t)} t h\Big(t, \frac{n_1}{q}, \frac{n_2 n_3 d^2}{q^2}\Big) \, dt,
\end{equation}
with $J^-_0$ defined similarly with $\frac{J_{2it}(x)}{\cosh(\pi t)}$ replaced by $K_{2it}(x) \sinh(\pi t)$, where in both cases
\begin{equation}
\label{eq:hdef}
 h(t, y_1,y_2)   =  \exp(-(t/T)^2) \frac{(t^2+\frac14)}{T^2} V_{1}(y_1, t) V_{2}(y_2,t).
\end{equation}
Let $\mathcal{S}_{N_1, N_2, N_3, C}^{\pm}$ be defined by 
\begin{equation*}
 \frac{1}{C \sqrt{N_1 N_2 N_3}} \sum_{c \equiv 0 \shortmod{q}}  
 \sum_{n_1, n_2, n_3 }  \chi(n_1) \overline{\chi}(n_2 n_3)
 S_{\overline{\chi}^2}(n_1,n_2 n_3;c)
  J^{\pm}_0\Big(\frac{4 \pi \sqrt{n_1 n_2 n_3}}{c}, \cdot \Big).
\end{equation*}
\begin{myprop}\label{helperprop1}
Suppose that there exists $B>2$ such that $\mathcal{S}_{N_1, N_2, N_3, C}^{\pm} \ll_{\eps} T^B q^\eps$  for all $N_1,N_2,N_3,C$ satisfying \begin{equation}
\label{eq:NvariableSizes}
N_1 \ll_{\eps} (qT)^{1+\varepsilon} , \qquad N_2 N_3 \ll_{\eps} d^{-2} (qT)^{2+\varepsilon}, \qquad q \ll C \ll (qT)^{100}.
\end{equation}
Then Theorem \ref{helperthm1} holds.
\end{myprop}
\begin{proof}
Recall the even parity condition on the sums over newforms in \eqref{Mpmqchidef}. This condition can be detected by extending the sums to all newforms and inserting the indicator function $\frac{1}{2} (1 + \lambda_f(-1))$ for Maass forms and Eisenstein series. By \eqref{eq:Kuznetsov}, we have
\begin{equation*}
 \mathcal{M}(q,\chi) = \mathcal{D} + \tfrac12 \mathcal{S}^{+} + \tfrac12 \mathcal{S}^{-},
\end{equation*}
where $\mathcal{D}$ is the diagonal term, and 
\begin{equation}\label{eq:cSpm}
 \mathcal{S}^{\pm} = 
 \sum_{(d,q) = 1} \frac{4}{d} 
 \sum_{n_1, n_2, n_3} \frac{\chi(n_1) \overline{\chi}(n_2 n_3)}{\sqrt{n_1 n_2 n_3}} 
 \sum_{c \equiv 0 \shortmod{q}} \frac{S_{\overline{\chi}^2}(\pm n_1, n_2 n_3 ; c)}{c} g^{\pm}\Big(\frac{4 \pi \sqrt{n_1 n_2 n_3}}{c}\Big).
\end{equation}
Here $g^{\pm}(x)$ is defined by \eqref{eq:gplusdef} and \eqref{eq:gminusdef} with respect to $h(t,\frac{n_1}{q}, \frac{n_2n_3d^2}{q^2})$ defined in \eqref{eq:hdef}.

The function $h$ is a valid test function for the hypotheses in the Bruggeman-Kuznetsov formula, and one may derive a crude bound of the form  $g^{\pm}(x) \ll x^{1-\varepsilon} T^{1+\varepsilon}$, as we will show in Section \ref{section:WeightFunctions1}. Hence by the Weil bound (see e.g.~\cite[Thm.\ 9.2]{KLkuznetsov}, which gives $|S_{\psi}(a,b ; c)| \leq d(c) (a,b,c)^{1/2} c^{1/2} q^{1/2}$, where $\psi$ has conductor $q|c$ and $d(\cdot)$ is the divisor function), we have that the sum over $c$ in \eqref{eq:cSpm} converges absolutely. We further develop the analytic properties of $g^{\pm}(x)$ in Section \ref{section:WeightFunctions1}. 

It is easy to see that $\mathcal{D} \ll_{\eps} T^{2+\varepsilon} q^{\varepsilon}$, and so the proof of Proposition \ref{helperprop1} reduces to showing that $\mathcal{S}^{\pm} \ll_{\eps} T^{B} q^{\varepsilon}$.

Next we apply a dyadic partition of unity to each of $n_1, n_2, n_3, c$.  Consider the component $w_0(\cdot)$ of this partition of unity which localizes the variables by $n_j \asymp N_j$, $c \asymp C$.  We may assume the inequalities \eqref{eq:NvariableSizes} hold, since if they do not, then the contribution from that piece of the partition of unity is small by trivial bounds. Hence, 
\begin{equation}
\label{eq:Sdef}
 \mathcal{S}^{\pm} = \sum_{(d,q)=1} \frac{4}{d} \sum_{N_1, N_2, N_3, C} \mathcal{S}_{N_1, N_2, N_3, C}^{\pm} + O_{\eps}((qT)^{\varepsilon}),
\end{equation}
where $N_1,N_2,N_3,C$ run over dyadic number satisfying the bounds \eqref{eq:NvariableSizes}. From the hypothesis on $\mathcal{S}_{N_1, N_2, N_3, C}^{\pm}$ in the statement of the proposition, we conclude the proof.
\end{proof}

\subsection{Poisson summation}
Let $m_1,m_2,m_3 \in \Z$ and $c>0$. Let $G = G(m_1,m_2,m_3 ; c)$ be the character sum defined by 
\begin{equation*}
 G = c^{-3} 
 \sumstar_{y \shortmod{c}}
 \sum_{x_1, x_2, x_3 \shortmod{c}}  \chi(x_1) \overline{\chi}(x_2 x_3) \chi^2(y)
 e_c(m_1 x_1 + m_2 x_2 + m_3 x_3 + x_1 y + x_2 x_3 \overline{y}),
\end{equation*}
where $e_c(x) = e(x/c)$. 
Let $M_1,M_2,M_3>0$ and let $w (\cdot) =w(n_1,n_2,n_3,c, m_1,m_2,m_3)$ be a family of $1$-inert functions (depending on $q,T,N_j,C, M_j$) with dyadic support on $n_j \asymp N_j$, $c \asymp C$, and $m_j \asymp M_j$. 
 Let $J^\pm(x,\cdot) = J^\pm(x, n_1, n_2,n_3, m_1,m_2,m_3, c)$ be defined by
 \begin{equation}\label{eq:Jpmdef}
J^{+}(x, n_1, n_2,n_3, m_1,m_2,m_3, c) =  w(\cdot) \intR \frac{J_{2it}(x)}{\cosh( \pi t)} t h\Big(t, \frac{n_1}{q}, \frac{n_2 n_3 d^2}{q^2}\Big) \, dt,
\end{equation}
and $J^-$ defined similarly with $K_{2it}(x) \sinh(\pi t)$ in place of $\frac{J_{2it}(x)}{\cosh(\pi t)}$.
Note that $J^\pm$ is identical to $J_0^\pm$ except that $w_0(\cdot)$ is replaced by $w(\cdot)$, which depends on the additional variables $m_1, m_2, m_3$). 

Let 
\begin{equation}
\label{eq:Kdef1}
 K^{\pm}_0 = \int_{\mathbb{R}^3}  J^{\pm}_0\Big(\frac{4 \pi \sqrt{t_1 t_2 t_3}}{c}, t_1, t_2, t_3, c\Big)
 e_c(-m_1 t_1 - m_2 t_2 - m_3 t_3) dt_1 dt_2 dt_3,
\end{equation}
and 
\begin{equation}
\label{eq:Kdef}
 K^{\pm} = \int_{\mathbb{R}^3}  J^{\pm}\Big(\frac{4 \pi \sqrt{t_1 t_2 t_3}}{c}, t_1, t_2, t_3, \cdot \Big)
 e_c(-m_1 t_1 - m_2 t_2 - m_3 t_3) dt_1 dt_2 dt_3.
\end{equation}
Finally, let $\epsilon_1,\epsilon_2,\epsilon_3 \in \{\pm 1\}$,
\begin{equation}\label{eq:Tpmdef}
\mathcal{T}^\pm = \mathcal{T}^\pm_{\epsilon_1,\epsilon_2, \epsilon_3} = \frac{1}{C \sqrt{N_1N_2N_3}} \sum_{c \equiv 0 \shortmod q} \sum_{m_j \epsilon_j \geq 1} G(m_1,m_2,m_3;c)K^\pm(m_1,m_2,m_3,c), 
\end{equation}
and
\begin{equation}\label{eq:T0pmdef}
\mathcal{T}^\pm_0 = \frac{1}{C \sqrt{N_1N_2N_3}} \sum_{c \equiv 0 \shortmod q} \sum_{m_1m_2m_3=0} G(m_1,m_2,m_3;c)K^\pm_0(m_1,m_2,m_3,c).
\end{equation}
\begin{myprop}\label{helperprop2}
Suppose that $\mathcal{T}^\pm, \mathcal{T}_0^\pm  \ll_\eps T^B q^\eps$ for some $B>2$ and for all $N_1,N_2,N_3,C$ satisfying \eqref{eq:NvariableSizes} and all $M_1,M_2,M_3$ satisfying $M_j \ll (qT)^A$ for some large but fixed $A$. Then $\mathcal{S}_{N_1, N_2, N_3, C}^{\pm}  \ll_\eps T^B q^\eps$ for all such $N_1,N_2,N_3,C$.
\end{myprop}
Sections \ref{sec:Gcalc}-\ref{section:Finale} are dedicated to the proof of the bounds $\mathcal{T}^\pm, \mathcal{T}_0^\pm  \ll_\eps T^B q^\eps$, which by Propositions \ref{helperprop2} and \ref{helperprop1} will finish the proof of Theorem \ref{helperthm1}, and hence of Theorem \ref{thm:mainthmMaassEisenstein}.
\begin{proof}
Applying Poisson summation in each of the variables $n_1, n_2, n_3$ modulo $c$  gives
\begin{equation}
\label{eq:SformulaAfterPoisson}
\mathcal{S}_{N_1,N_2,N_3,C}^{\pm} = \frac{1}{C \sqrt{N_1 N_2 N_3}}  \sum_{c \equiv 0 \shortmod{q}} \sum_{m_1, m_2, m_3 \in \mz}
 \chi(\pm 1)G(\pm m_1, m_2, m_3;c) K^{\pm}_0(m_1, m_2, m_3, c).
 \end{equation}
 By integrating $K_0^\pm$ by parts three times in each variable, we have by \eqref{eq:NvariableSizes} a crude bound of the form \begin{equation}\label{eq:K0pmTrivbound}K_0^\pm(m_1,m_2,m_3,c) \ll (qT)^{A} \prod_{j=1}^3(1+|m_j|)^{-3},\end{equation} for some possibly large but fixed $A$. Therefore the sum \eqref{eq:SformulaAfterPoisson} converges absolutely, and we may in fact truncate each $m_j$ variable at $|m_j| \ll (qT)^{A'}$ for some large $A'$ depending polynomially on $1/\eps$ at the cost of a small error term.
 
Next, we separate the terms with $m_1m_2m_3=0$ in $\mathcal{S}_{N_1,N_2,N_3,C}^{\pm}$ from those in which none of the $m_j$ vanish. The terms with $m_1m_2m_3=0$ form the sum $\mathcal{T}_0^\pm$ defined in \eqref{eq:T0pmdef}. Leaving these terms aside, we split the remaining terms for which $m_j \neq 0$ for all $j$ into eight separate sums according to the octants of $\Z^3-\{m_1m_2m_3=0\}$.  Let us parametrize these eight sums by $(\epsilon_1,\epsilon_2,\epsilon_3) \in \{\pm1\}^3$. The octant corresponding to $\epsilon_1,\epsilon_2, \epsilon_3$ is then described by the inequalities $m_j \epsilon_j \geq 1$ for $j=1,2,3$. Given one choice of signs $\epsilon_j \in \{\pm 1\}$, we insert a dyadic partition of unity to the $m_1,m_2, m_3$ sums, which localizes each $|m_j|\asymp M_j \ll (qT)^{A'}$. The result of all of these decompositions is that 
\begin{equation*}
\mathcal{S}_{N_1,N_2,N_3,C}^{\pm} =  \mathcal{T}^\pm_0 + \sum_{M_1,M_2,M_3} \sum_{\epsilon_1,\epsilon_2,\epsilon_3 \in \{\pm 1\}} \mathcal{T}_{\epsilon_1, \epsilon_2, \epsilon_3}^\pm + O_\eps((qT)^\eps).
\end{equation*}
The proposition now follows from the hypothesized bounds on $\mathcal{T}_0^\pm$ and $\mathcal{T}^\pm$. 
\end{proof}

The main focus in this paper is on the character sum $G$, which is a generalization of the character sum found in the previous works \cite{CI} \cite{YoungCubic} \cite{Petrow} \cite{PetrowYoung}, since $\chi$ is no longer assumed to be quadratic and $q$ is not necessarily square-free. On the other hand, $K^\pm$ is very similar in shape to the oscillatory integrals found in the above references, so 
in Section \ref{section:WeightFunctions2}
we largely quote the existing literature.

\section{The calculation of $G$}\label{sec:Gcalc}
Based on the structural approach presented in \cite{PetrowYoung}, our primary goal on the arithmetical aspects of $G$ is to understand the analytic properties of the Dirichlet series
\begin{equation}\label{Zdef}
Z(s_1, s_2, s_3, s_4) := \sum_{\epsilon_1m_1, \epsilon_2 m_2, \epsilon_3m_3 \geq 1} \sum_{c \equiv 0 \shortmod{q}} \frac{cq G(m_1, m_2, m_3;c) e_c(-m_1 m_2 m_3) \chi(-1)}{m_1^{s_1} m_2^{s_2} m_3^{s_3} (c/q)^{s_4}}.
\end{equation}
For simplicity of notation, we only consider the case of \eqref{Zdef} where $\epsilon_j = 1$ for all $j$, since the other sign combinations can be treated in the same way. Of course, we cannot neglect to study the contribution from  $m_1m_2m_3=0$ as well.
In any event, we calculate $G$ in explicit form as much as possible.

\subsection{Simplifications}
Write $c=qr$ with $r\geq 1$. 
We have
\begin{equation*}
 \sum_{x_1 \shortmod{c}} \chi(x_1) e_c(m_1 x_1 + x_1 y) = r \tau(\chi) \overline{\chi}\Big(\frac{m_1 + y}{r}\Big),
\end{equation*}
where the sum vanishes unless $y \equiv - m_1 \pmod{r}$, and $\tau(\chi)$ denotes the Gauss sum of $\chi \pmod{q}$.  Similarly, we calculate the $x_2$ sum by
\begin{equation*}
 \sum_{x_2 \shortmod{c}} \overline{\chi}(x_2) e_c(x_2(m_2 + x_3 \overline{y})) = r \tau(\overline{\chi}) \chi\Big(\frac{m_2 + x_3 \overline{y}}{r}\Big),
\end{equation*}
where the sum vanishes unless $x_3  \equiv - m_2 y  \pmod{r}$.  Changing variables $x_3 \rightarrow y x_3$, we hence obtain
\begin{equation*}
 G = \frac{r^2 \tau(\chi) \tau(\overline{\chi})}{c^3} 
 \sumstar_{\substack{y \shortmod{c} \\ y \equiv - m_1 \shortmod{r}}}
 \sum_{\substack{x_3 \shortmod{c} \\ x_3 \equiv - m_2 \shortmod{r}}} \overline{\chi}(x_3) \chi(y)
 e_c( m_3 y x_3)
 \overline{\chi}\Big(\frac{m_1 + y}{r}\Big)
 \chi\Big(\frac{m_2 + x_3}{r}\Big).
\end{equation*}
Since $(y,c) = 1$ we learn that $G=0$ unless
\begin{equation}
\label{eq:m1iscoprimetor}
 (m_1, r) = 1.
\end{equation}
Provided we maintain this condition, we can drop the condition that $(y,c) = 1$.  Writing $y = -m_1 + r u$ and $x_3 = - m_2 + r t$, we obtain
\begin{equation*}
 G(m_1, m_2, m_3;c) = c^{-3} r^2 \tau(\chi) \tau(\overline{\chi}) e_c(m_1 m_2 m_3)  H_{\chi}(m_1, m_2, m_3, r) \delta_{(m_1,r)=1},
\end{equation*}
where
\begin{multline*}
H_{\chi}(m_1, m_2, m_3, r)  = \sum_{u,t \shortmod{q}} \chi(t)  \overline{\chi}(u)
\overline{\chi}(- m_2 + r t) \chi(-m_1 + r u) \\
\times e_c( m_3 (-m_1 + r u) (- m_2 + r t) - m_1 m_2 m_3).
\end{multline*}
Note that
\begin{equation}
\label{eq:GHchirelation}
\delta_{(m_1,r)=1} H_{\chi}(m_1, m_2, m_3, r) = cq G(m_1, m_2, m_3;c) e_c(-m_1 m_2 m_3) \chi(-1),
\end{equation}
so that
\begin{equation}
\label{eq:Zdef}
 Z(s_1, s_2, s_3, s_4) =  \sum_{\substack{m_1, m_2, m_3, r \geq 1 \\ (m_1,r)=1}}  \frac{H_{\chi}(m_1, m_2, m_3, r)}{ m_1^{s_1} m_2^{s_2} m_3^{s_3} r^{s_4}}.
\end{equation}

Next we derive some simple but useful symmetries of $H_{\chi}$.  Although we only need to compute $H_{\chi}(m_1, m_2, m_3, r)$ when \eqref{eq:m1iscoprimetor} holds, it will be more convenient not to assume this condition.
Changing variables $t \rightarrow (-m_1 + ru)^{-1} t$ gives
\begin{equation*}
H_{\chi}(m_1, m_2, m_3, r)= \sum_{u,t \shortmod{q}} \chi(t) \chi(-m_1 + ru) \overline{\chi}(u)
\overline{\chi}(r t - m_2(-m_1 + ru)) e_q(m_3 t - m_2 m_3  u ).
\end{equation*}
Next shift by $t \rightarrow t + m_2  u$, giving
\begin{equation}
\label{eq:HrGeneralFormula}
H_{\chi}(m_1, m_2, m_3, r)=  \sum_{u,t \shortmod{q}} 
\chi(t+m_2u) \overline{\chi}(r t + m_1 m_2)
\overline{\chi}(u) \chi(-m_1 + ru)
 e_q(m_3 t ).
\end{equation}
Since $G(m_1, m_2, m_3;c)$ is symmetric in $m_2,m_3$, we see that
\begin{equation}
\label{eq:Hchisymmetry}
 H_{\chi}(m_1, m_2, m_3,r) = H_{\chi}(m_1, m_3, m_2,r).
\end{equation}
Observe that if $(m_1 m_2, r, q) \neq 1$ then every summand in \eqref{eq:HrGeneralFormula} vanishes.
Together with the symmetry \eqref{eq:Hchisymmetry}, we obtain
\begin{equation}
\label{eq:Hchivanishesunlessm1m2m3randqarecoprime}
 H_{\chi}(m_1, m_2, m_3, r) = 0 \quad
 \text{ if }
 (m_1 m_2 m_3, r, q) \neq 1.
\end{equation}

If $(q,r)=1$ then there is some additional symmetry. We claim that 
\begin{equation}
\label{eq:Hchisymmetry2}
H_{\chi}(m_1, m_2, m_3,r) = H_{\overline{\chi}}(m_2, m_1, m_3, r) \quad \text{ if } (q,r)=1.  
\end{equation}
Indeed, changing variables $t \rightarrow \overline{r} t$, $u \rightarrow \overline{r} u$, gives
\begin{equation*}
H_{\chi}(m_1, m_2, m_3, r)=\sum_{u,t \shortmod{q}} 
\chi(t+m_2u) \overline{\chi}(  t + m_1 m_2)
\overline{\chi}(u) \chi(-m_1 +  u) 
 e_q(m_3 \overline{r} t ).
\end{equation*}
Next we change variables $u \rightarrow u+m_1$, followed by $t \rightarrow ut- m_1 m_2$ (note $u$ is coprime to $q$ for every non-zero summand), giving
\begin{equation}
\label{eq:Hchiqcoprimetor}
 H_{\chi}(m_1, m_2, m_3, r)= e_q(-m_1 m_2 m_3 \overline{r}) \sum_{u,t \shortmod{q}} 
\chi(t+m_2)  \overline{\chi}(t)
\overline{\chi}(u+m_1) \chi(u) 
 e_q(m_3 \overline{r} ut ),
\end{equation}
from which we deduce \eqref{eq:Hchisymmetry2}.

\subsection{Decomposition into Dirichlet characters}
\label{section:DirichletDecomposition}
It is possible to calculate $H_{\chi}$ further, as in \cite{CI},  but going to the Fourier transform of $H_{\chi}$ turns out to be a more advantageous move.

Begin by writing $r = r_0 r'$ and $m_j = m_{j,0} m_j'$, $j=1,2,3$, with 
\begin{equation}
\label{eq:variablesdividingqinfinity}
 m_{j,0} | q^{\infty}, \qquad r_0|q^{\infty}
\end{equation}
and $(m_1' m_2' m_3' r', q) = 1$.  
Inside the expression \eqref{eq:HrGeneralFormula}, change variables $t \rightarrow m_1' m_2'\overline{r'} t$ and $u \rightarrow \overline{r'} m_1' u$, giving
\begin{equation*}
 H_{\chi}(m_1, m_2, m_3, r) = H_{\chi}(m_{1,0}, m_{2,0}, m_{3,0} w, r_0),
\end{equation*}
where
\begin{equation*}
 w = m_1' m_2' m_3' \overline{r'}.
\end{equation*}
Note that $(w,q) = 1$ by assumption.  We may then view 
$H_{\chi}$ as a function of $w$ on $(\mz/q\mz)^{\times}$, and apply multiplicative Fourier analysis.  That is, we write
\begin{equation}
\label{eq:HchiFourierInversion}
 H_{\chi}(m_{1,0}, m_{2,0}, m_{3,0} w, r_0)
 = \frac{1}{\varphi(q)} \sum_{\psi \shortmod{q}} 
 \widehat{H}(\psi) \psi(w),
\end{equation}
where
\begin{equation}
\label{eq:HhatDef1}
 \widehat{H}(\psi) = \widehat{H} = \widehat{H}(\psi, \chi, m_{1,0}, m_{2,0}, m_{3,0}, r_0) =
 \sum_{v \shortmod{q}} 
 H_{\chi}(m_{1,0}, m_{2,0}, m_{3,0} v, r_0) \overline{\psi}(v).
 \end{equation}
 Expanding the definition, we have
\begin{equation} 
\label{eq:HhatDef2}
 \widehat{H}(\psi,\chi,m_1,m_2,m_3,r) =  
 \sum_{t,u,v \shortmod{q}}
 \chi(t+ m_2 u) \overline{\chi}(r t + m_{1} m_{2})
\overline{\chi}(u) \chi(-m_{1} + r u) 
 e_q(m_{3} v t ) \overline{\psi}(v).
\end{equation}
The sum $\widehat{H}(\psi)$ inherits from \eqref{eq:Hchisymmetry} and \eqref{eq:Hchisymmetry2} the symmetries
\begin{equation}
\label{eq:HchiHatSymmetry}
\begin{split}
 \widehat{H}(\psi,\chi,m_1, m_3, m_2, r) &= \widehat{H}(\psi,\chi,m_1, m_2, m_3, r) \\
  \widehat{H}(\psi,\chi,m_2, m_1, m_3, r) &= \widehat{H}(\psi,\overline{\chi},m_1, m_2, m_3, r), \qquad \text{if } (q,r) = 1.
  \end{split}
\end{equation}
Similarly, from \eqref{eq:Hchivanishesunlessm1m2m3randqarecoprime} we deduce
\begin{equation}
\label{eq:Hhatchivanishesunlessm1m2m3randqarecoprime}
 \widehat{H}(\psi, \chi, m_1, m_2, m_3, r) = 0 \qquad
 \text{if}
 \qquad
 (m_1 m_2 m_3, r, q) \neq 1.
\end{equation}

We immediately see the pleasant factorization
\begin{equation}
\label{eq:ZintermsofDirichletLandZfin}
 Z(s_1, s_2, s_3, s_4)
 = \frac{1}{\varphi(q)} \sum_{\psi \shortmod{q}}
 \frac{L(s_1, \psi) L(s_2, \psi) L(s_3, \psi) L(s_4, \overline{\psi})}{\zeta^{(q)}(s_1 + s_4)} Z_{\text{fin}},
\end{equation}
where
\begin{equation*}
 Z_{\text{fin}} = Z_{\text{fin}}(s_1,s_2,s_3,s_4) := \sum_{\substack{m_{1,0}, m_{2,0}, m_{3,0}, r_0 | q^{\infty} \\ (m_{1,0},r_0)=1}} \frac{\widehat{H}(\psi, \chi, m_{1,0}, m_{2,0}, m_{3,0}, r_0)}{m_{1,0}^{s_1} m_{2,0}^{s_2}m_{3,0}^{s_3} r_0^{s_4} }.
\end{equation*}
The factor $\zeta^{(q)}(s_1 + s_4)^{-1}$ arose from M\"{o}bius inversion to detect $(m_1', r') = 1$.

Now the task is to understand the analytic properties of $Z_{\text{fin}}$.
Suppose $q = q_1 q_2$ with $(q_1, q_2) = 1$, $\chi = \chi_1 \chi_2$ and  $\psi = \psi_1 \psi_2$ with $\chi_j$, $\psi_j$ modulo $q_j$.  Similarly, write $a = a_1 a_2$, and so on with $b,c,d$.  By the Chinese remainder theorem, we have
that $\widehat{H}(\psi, \chi, a,b,c,d)$ factors as a sum of modulus $q_1$ times a sum of modulus $q_2$.  The sum modulo $q_1$ equals
\begin{equation*}
  \sum_{t_1, u_1, v_1 \shortmod{q_1}}
 \chi_1(t_1+ b_1 b_2 u_1) \overline{\chi_1}(d_1 d_2 t_1 + a_1 a_2 b_1 b_2)
\overline{\chi_1}(u_1) \chi_1(-a_1 a_2 + d_1 d_2 u_1) 
 e_{q_1}(\overline{q_2} c_1 c_2 v_1 t_1 ) \overline{\psi_1}(v_1).
 \end{equation*}
Changing variables $t_1 \rightarrow a_2 b_2 \overline{d_2} t_1$, $u_1 \rightarrow a_2 \overline{d_2} u_1$, and $v_1 \rightarrow \overline{a_2 b_2 c_2} d_2 q_2 v_1$
shows 
\begin{equation*}
 \widehat{H}(\psi, \chi, a,b,c,d) = \epsilon
\widehat{H}(\psi_1, \chi_1, a_1,b_1,c_1,d_1)
  \widehat{H}(\psi_2, \chi_2, a_2,b_2,c_2,d_2),
\end{equation*}
where $\epsilon = \psi_1(a_2 b_2 c_2 \overline{q_2 d_2}) 
 \psi_2(a_1 b_1 c_1 \overline{q_1 d_1}) $.
Pleasantly, $\widehat{H}$ is almost multiplicative in terms of $\chi, \psi$, and the only ``twisted'' aspect comes from the factor $\epsilon$.

This shows
\begin{equation}
\label{eq:ZfinFormula}
 Z_{\text{fin}} = \omega
 \prod_{p^j || q}
 \sum_{\substack{a,b,c,d | p^{\infty} \\ (a,d) = 1}}
 \frac{\eta(a b c) \overline{\eta}(d)}{a^{s_1} b^{s_2} c^{s_3} d^{s_4} } \widehat{H}(\psi_p, \chi_p, a,b,c,d),
\end{equation}
where $\eta$ is some Dirichlet character depending on $\psi$ and $p$, and $\omega$ is some complex number of absolute value $1$, which depends on $\psi$.  Here $\chi_p, \psi_p$ are the $p$-parts of $\chi,\psi$.

\section{Evaluation of $\widehat{H}$}
Here we comprehensively evaluate $\widehat{H}$ when $q=p^k$, $k\geq 1$.  
Recall that $\widehat{H}$ was defined in \eqref{eq:HhatDef2}. Throughout this section we assume that $m_1,m_2, m_3, r | q^{\infty}$.

\subsection{Elementary lemmas on character sums}
We begin with some character sum evaluations that are used repeatedly in the calculations of $\widehat{H}$.

\begin{mylemma}
\label{lemma:chisumprimitive}
 Suppose that $\chi$ is primitive modulo $q$ and $d|q$, $d \neq q$.  Then
 \begin{equation*}
  \sum_{\substack{a \shortmod{q} \\ a \equiv b \shortmod{d}}} \chi(a) = 0.
 \end{equation*}
\end{mylemma}
This well-known lemma may be found in 
\cite[(3.9)]{IK},
for instance.

\begin{mylemma}
\label{lemma:chisumLFTpartialSumpSquared}
 Suppose $p$ is prime, $a \in \mz$, and $\chi$ has conductor $p^k$, $k \geq 2$.  Then
 \begin{equation*}
  \sum_{\substack{t \shortmod{p^{k}} \\ t \equiv a \shortmod{p}}} \chi(t) \overline{\chi}(t+1) = 0.
 \end{equation*}
\end{mylemma}
\begin{proof}
 If $(a(a+1),p) \neq 1$ the sum is empty, so suppose otherwise.  Then from $\chi(t) \overline{\chi}(t+1) = \overline{\chi}(1+\overline{t})$, and changing variables $t \rightarrow \overline{t}$, the conclusion follows from Lemma \ref{lemma:chisumprimitive}.
\end{proof}

\begin{mylemma}
\label{lemma:LFTprimeSquared}
 Suppose $\chi$ is primitive modulo $q$, and let $a,b,c,d \in \mz$ with $(a,c,q) = 1$.  Then
 \begin{equation}
 \label{eq:LFTsumprimeSquared}
  \sum_{t \shortmod{q}} \chi(at+b) \overline{\chi}(ct+d) 
  = \chi(a) \overline{\chi}(c) R_{q}(ad-bc),
 \end{equation}
 where $R_q(n) = S(n,0;q)$ is the Ramanujan sum.
\end{mylemma}

\begin{proof}
We first claim the sum vanishes unless $(a,q) = (c,q)=1$.  By symmetry, suppose $(a,q) \neq 1$.  Then $\chi(at+b)$ is constant for $t$ ranging over an arithmetic progression modulo $\frac{q}{(a,q)}$.  Lemma \ref{lemma:chisumprimitive} shows that the sum over this arithmetic progression of $\overline{\chi}(ct+d)$ vanishes unless $q|c \frac{q}{(a,q)}$, i.e. $(a,q) | c$, whence $1=(a,c,q) = (a,q)$, contradiction.  Therefore, \eqref{eq:LFTsumprimeSquared} is derived if $(a,q) \neq 1$ or $(c,q) \neq 1$.

Now suppose $(a,q) = (c,q) = 1$.  By converting to additive characters, that is, using 
\begin{equation}
\label{eq:changeofbasisformula}
\chi(at+b) = \frac{1}{\tau(\overline{\chi})} \sum_{x \shortmod{q}} \overline{\chi}(x) e_q(x(at+b)),
\end{equation}
and likewise for $\overline{\chi}(ct+d)$, the
formula \eqref{eq:LFTsumprimeSquared} follows from a routine calculation.
\end{proof}

\subsection{The case $\psi$ primitive modulo $q$}
\begin{mylemma}
\label{lemma:HhatEvaluationpsiPrimitive}
 Suppose $p$ is a prime and $q=p^k$, $k\geq 1$. Suppose $\psi$ is primitive modulo $q$.  Then $\widehat{H}$ 
 vanishes unless $(m_1 m_2 m_3 r, q) = 1$, in which case
 \begin{equation}
 \label{eq:Hhatintermsofgchipsi}
  \widehat{H}(\psi,\chi, 1, 1, 1, 1) = \tau(\overline{\psi}) g(\chi, \psi),
 \end{equation}
 where $g(\chi,\psi)$ was defined by \eqref{eq:gdef}, and $\tau(\overline{\psi})$ is the Gauss sum.
\end{mylemma}

\begin{proof}
 Since $\psi$ is primitive, the sum over $v$ in \eqref{eq:HhatDef2} is a Gauss sum, giving
\begin{equation*}
\widehat{H}(\psi,\chi, m_1, m_2, m_3, r) = \tau(\overline{\psi}) 
 \sum_{t,u \shortmod{q}}
 \chi(t+m_{2} u) \overline{\chi}(r  t + m_{1 } m_{2 })
\overline{\chi}(u) \chi(-m_{1 } + r  u) 
 \psi(m_3 t).
\end{equation*} 
Hence, $\widehat{H}$ vanishes unless $(m_3,q) = 1$.  By the first symmetry in \eqref{eq:HchiHatSymmetry}, this means it vanishes unless $(m_2, q) = 1$, too.  We claim that it vanishes unless $(m_1, q) = 1$.  
If $p|m_1$ and $(p,r) = 1$ then the claim follows from the second symmetry in \eqref{eq:HchiHatSymmetry}, while if $p|(m_1, r)$ then the claim follows from \eqref{eq:Hhatchivanishesunlessm1m2m3randqarecoprime}.
Thus we may set $m_1 = m_2 = m_3 = 1$, since we have assumed that $m_1m_2m_3r|q^\infty$.

If $(p,r) = 1$, then $r =1$, in which case
\begin{equation*}
\widehat{H}(\psi,\chi,1,1,1,1)= \tau(\overline{\psi}) \sum_{t,u \shortmod{q}} \chi(t+u) \overline{\chi}(t+1) \overline{\chi}(u) \chi(u-1) \psi(t).
\end{equation*}
Changing variables $u \rightarrow u+1$ followed by $t \rightarrow ut -1$, and finally changing the roles of $u$ and $t$ (for cosmetic purposes), we obtain \eqref{eq:Hhatintermsofgchipsi}.

Finally, suppose that $p|r$ and $m_1  m_2  m_3 = 1$.  Changing variables $t \rightarrow ut$ gives
\begin{equation*}
 \widehat{H}(\psi,\chi,1,1,1,r) = 
 \tau(\overline{\psi}) 
 \sum_{t,u \shortmod{q}}
 \chi(1+t) \chi(-1 + r  u) \overline{\chi}(1+r u t)
 \psi(tu).
\end{equation*}
Since $\chi(-1+ru) \overline{\chi}(1+rut)$ is periodic in $u$ with period $\frac{p^k}{(r,p^k)} \leq p^{k-1}$, the sum over $u$ vanishes by Lemma \ref{lemma:chisumprimitive}, since $\psi$ has conductor $p^k$.
\end{proof}

\subsection{The case of $\psi$ trivial}
\begin{mylemma}
\label{lemma:HhatEvaluationPsiTrivial}
 Suppose $\psi = \chi_0$ is the trivial character, and $q=p^k$, $k \geq 1$.  Then
 \begin{equation*}
  \widehat{H}(\chi_0,\chi,m_1, m_2, m_3, r) = 
  \chi_0(r) R_q(m_1) R_q(m_2) R_q(m_3)
  + q R_q(r) \chi(-1) \chi_0(m_1 m_2 m_3).
 \end{equation*}
\end{mylemma}

\begin{proof}
In this case, $\widehat{H}(\chi_0,\chi,m_1,m_2,m_3,r)$ equals
\begin{equation*}
\sum_{t,u \shortmod{q}}
 \chi(t+m_{2} u) \overline{\chi}(r  t + m_{1 } m_{2 })
\overline{\chi}(u) \chi(-m_{1 } + r  u) 
 R_q(m_3 t).
\end{equation*}
Write $R_q(m_3 t) = R_q(m_3) + (R_q(m_3 t) - R_q(m_3))$, and note that if $p \nmid t$ then $R_q(m_3 t) - R_q(m_3) = 0$.  We accordingly write $\widehat{H} = S_1 + S_2$ where
\begin{equation*}
 S_1 = R_q(m_3) \sum_{t,u \shortmod{q}}
 \chi(t+m_{2} u) \overline{\chi}(r  t + m_{1 } m_{2 })
\overline{\chi}(u) \chi(-m_{1 } + r  u),
 \end{equation*}
and $S_2 = \widehat{H} - S_1$.
 We will show
 \begin{equation*}
  S_1 = \chi_0(r) R_q(m_1) R_q(m_2) R_q(m_3), \quad \text{and} \quad
  S_2 = q R_q(r) \chi(-1) \chi_0(m_1 m_2 m_3).
 \end{equation*}

 First we evaluate $S_1$.  By Lemma \ref{lemma:LFTprimeSquared} and since we may assume $(m_1 -ru, q) = 1$, we have
 \begin{equation*}
  \sum_{t \shortmod{q}} \chi(t+m_2 u) \overline{\chi}(rt+m_1 m_2) = \overline{\chi}(r) R_q(m_2).
 \end{equation*}
 To finish the evaluation of $S_1$, we apply Lemma \ref{lemma:LFTprimeSquared} to give
\begin{equation*}
 \sum_{u \shortmod{q}} \overline{\chi}(u) \chi(ru-m_1) = \chi(r) R_q(m_1),
\end{equation*}

Now we evaluate $S_2$.  The $t$-sum is restricted by $p|t$, and so we see that $S_2$ vanishes unless $(p,m_1m_2)=1$. By our convention, we may set $m_1 = m_2 = 1$, giving
\begin{equation*}
 S_2 = \chi_0(m_1 m_2) \chi(-1) \sum_{\substack{t \shortmod{q} \\ p|t }} \sum_{u \shortmod{q}}
 \chi(t+  u) \overline{\chi}(r  t + 1)
\overline{\chi}(u) \chi(1 - r  u)
(R_q(m_3 t) - R_q(m_3)).
\end{equation*}
Next we change variables $t \rightarrow ut$, giving
\begin{equation*}
 S_2  = \chi_0(m_1 m_2) \chi(-1) \sum_{\substack{t \shortmod{q} \\ p|t }} 
 (R_q(m_3 t) - R_q(m_3)) \chi(t+  1)
 \sumstar_{u \shortmod{q}}
  \overline{\chi}(r u t + 1)
\chi(1 - r  u).
\end{equation*}
For the inner sum over $u$, apply $u \rightarrow u^{-1}$, giving
\begin{equation*}
 \sumstar_{u \shortmod{q}} \chi(1-u^{-1} r) \overline{\chi}(1+ u^{-1} rt) = \sum_{u \shortmod{q}} \chi(u-r) \overline{\chi}(u+rt),
\end{equation*}
where we could omit the condition $(u,q) = 1$ since $p|t$. By Lemma \ref{lemma:LFTprimeSquared}, this equals $R_q(r(t+1)) = R_q(r)$, provided $(t+1, q) = 1$.  Hence
\begin{equation*}
 S_2 = \chi_0(m_1 m_2) \chi(-1) R_q(r) \sum_{\substack{t \shortmod{q} \\ p|t }} \chi(t+  1)
 (R_q(m_3 t) - R_q(m_3)) .
\end{equation*}
To complete the proof, we will show
\begin{equation*}
 \sum_{\substack{t \shortmod{q} \\ p|t }} \chi(t+  1)
 (R_q(m_3 t) - R_q(m_3)) = q \chi_0(m_3).
\end{equation*}
If $q=p$, this is immediate, noting $R_p(0) - R_p(m_3) = p \chi_0(m_3)$, so suppose $q=p^k$, $k \geq 2$.  If $(p,m_3) = 1$ it is easy to verify the claim using the evaluation $R_q(m_3 t) = \sum_{d|(q,t)} d \mu(q/d)$ and
Lemma \ref{lemma:chisumprimitive}.  If $p|m_3$, then $R_q(m_3 t)$ is periodic (in $t$) of period $p^{k-1}$, so the sum vanishes by Lemma \ref{lemma:chisumprimitive}.
\end{proof}

\subsection{The case $q=p^k$, $\psi$ of conductor $p^j$, $1 \leq j < k$.}
\begin{myconj}
\label{conj:characterSumBoundIntermediateCase}
Suppose $\chi$ has conductor $p^k$, and $\psi$ has conductor $p^j$, with $1 \leq j < k$.  Then
\begin{equation}
\label{eq:characterSumBoundIntermediateCase}
 \sum_{u, y \shortmod{p^j}} \psi(uy) \chi(1+p^{k-j} y) \chi(1-p^{k-j} u) \overline{\chi}(1+uy p^{2(k-j)}) = O(p^j).
\end{equation}
\end{myconj}
\begin{mylemma}
 Conjecture \ref{conj:characterSumBoundIntermediateCase} holds in case $k=2, j=1$.
\end{mylemma}
\begin{proof}
By converting to additive characters (as in \eqref{eq:changeofbasisformula}), one may 
show $\sum_{x \mymod{p}} \psi(x) \chi(1+px) = \frac{\tau(\psi) \tau(\overline{\chi \psi})}{\tau(\overline{\chi})}$, which has absolute value $\sqrt{p}$.  In the case $k=2, j=1$ the factor $\chi(1+uy p^{2(k-j)})$ is identically $1$, and so \eqref{eq:characterSumBoundIntermediateCase} is the product of two sums of this type.
\end{proof}

In the following lemma and its proof, we use the convention that if $\chi$ is a Dirichlet character, $x \in \mq$, $x \not \in \mz$, then $\chi(x) = 0$. 
\begin{mylemma}
\label{lemma:HhatEvaluationq=pSquaredPsiConductorp}
 Let $\chi$, $\psi$ be as in Conjecture \ref{conj:characterSumBoundIntermediateCase}, and suppose \eqref{eq:characterSumBoundIntermediateCase} holds.  Then 
 \begin{equation}
 \label{eq:HhatEvaluationq=pSquaredPsiConductorp}
  \widehat{H}(\psi,\chi,m_1,m_2,m_3,r) = 
  \begin{cases}
   0, \quad & (m_1 m_2 m_3 r, p) = 1 \\
 \chi_0(\frac{r}{p^{k-j}}) O(p^{2k-\frac{j}{2}}), \quad & p|r, \thinspace m_1m_2m_3=1 \\
  \chi_0(\frac{m_1}{p^{k-j}}) 
   \chi_0(\frac{m_2}{p^{k-j}}) 
    \chi_0(\frac{m_3}{p^{k-j}}) 
  O(p^{3k-\frac{3j}{2}}), \quad &p|m_1 m_2 m_3, \thinspace r=1 \\
0, & p|r, \thinspace p | m_1m_2m_3.
  \end{cases}
 \end{equation}
 In particular, the bound \eqref{eq:HhatEvaluationq=pSquaredPsiConductorp} holds unconditionally for $k=2$, $j=1$.
 Furthermore, if in the second line of \eqref{eq:HhatEvaluationq=pSquaredPsiConductorp} the $O(p^{2k-\frac{j}{2}})$ is replaced by $O(p^{2k+\frac{j}{2}})$ and in the third line the $O(p^{3k-\frac{3j}{2}})$ is replaced by $O(p^{3k-\frac{j}{2}})$, then the revised bounds holds unconditionally for all $1 \leq j < k$ .
\end{mylemma}

\begin{proof}
We begin with the observation
\begin{equation}
\label{eq:gausssumnonprimitive}
 \sum_{v \shortmod{q}} e_q(m_3 v t) \overline{\psi}(v) = 
 p^{k-j} \tau(\overline{\psi})  \psi\Big(\frac{m_3 t}{p^{k-j}}\Big).
\end{equation}
 Using \eqref{eq:gausssumnonprimitive} in \eqref{eq:HhatDef2},
we have
\begin{equation}
\label{eq:HhatFormulaIntermediateCaseAfterGaussSumCalculated}
 \widehat{H} =  
 p^{k-j} \tau(\overline{\psi}) \sum_{t,u\shortmod{p^k}}
 \chi(t+ m_2 u) \overline{\chi}(r t + m_{1} m_{2})
\overline{\chi}(u) \chi(-m_{1} + r u) 
 \psi\Big(\frac{m_3 t}{p^{k-j}}\Big).
\end{equation}

First suppose that $(m_1 m_2 m_3 r, p) = 1$.  Then changing variables $t \rightarrow ut$, we have
\begin{equation*}
 \widehat{H}(\psi,\chi, 1,1,1,1) = 
 p^{k-j} \tau(\overline{\psi}) \sum_{t,u\shortmod{p^k}}
 \chi(t+ 1) \overline{\chi}(ut + 1)
 \chi(-1 +   u) \psi(u)
 \psi\Big(\frac{  t}{p^{k-j}}\Big)
 .
\end{equation*}
Note that $\overline{\chi}(ut+1) \psi(u)$ is periodic in $u$ of period $p^j$, since $p^{k-j} | t$ and $\psi$ has conductor $p^j$.
Hence by Lemma \ref{lemma:chisumprimitive} the sum over $u$ vanishes, as desired.

Now suppose $p|r$ and $m_1m_2m_3=1$. Then 
\begin{equation*}
 \widehat{H}(\psi,\chi,1,1,1,r) =  
 p^{k-j} \tau(\overline{\psi}) \sum_{t,u\shortmod{p^k}}
 \chi(t+  u) \overline{\chi}(r t + 1)
\overline{\chi}(u) \chi(-1 + r u) 
 \psi\Big(\frac{  t}{p^{k-j}}\Big).
\end{equation*}
Changing variables $t \rightarrow u p^{k-j} y$ (where $y$ now runs modulo $p^j$), we have
\begin{equation*}
\widehat{H}(\psi,\chi,1,1,1,r) = p^{k-j} 
 \tau(\overline{\psi}) 
 \sum_{y \shortmod{p^j}}
 \sumstar_{u\shortmod{p^k}}
  \chi(1+p^{k-j} y ) \chi(-1 + r u)
  \overline{\chi}(1+r p^{k-j} uy)
  \psi(uy)
.
\end{equation*}
We claim the $u$-sum vanishes unless $v_p(r) = k-j$, as we now show. Note that $\overline{\chi}(1+ r p^{k-j} uy) \psi(u)$ is periodic in $u$ with period $p^j$, while if $v_p(r) < k-j$ then $\chi(-1+ru)$ has period at least $p^{j+1}$.  Lemma \ref{lemma:chisumprimitive} then shows the claim. On the other hand, if $v_p(r) > k-j$, then $\chi(-1+ru) \overline{\chi}(1+ r p^{k-j} uy)$ is periodic with period $p^{j-1}$, while $\psi(u)$ has least period $p^j$.  Again, Lemma \ref{lemma:chisumprimitive} shows the claim.

Thus we may now restrict attention to $r=p^{k-j}$, in which case $\widehat{H}(\psi,\chi,1,1,1,p^{k-j})$ equals
\begin{equation*}
 p^{k-j} 
 \tau(\overline{\psi}) 
 \sum_{y \shortmod{p^j}}
 \sum_{u\shortmod{p^k}}
 \chi(1+p^{k-j} y ) \chi(-1 + p^{k-j} u)
  \overline{\chi}(1+  p^{2(k-j)} uy )
 \psi(uy)
.
\end{equation*}
The summand is periodic in $u$ modulo $p^j$, so it is the same sum repeated $p^{k-j}$ times.  The conjectured bound \eqref{eq:characterSumBoundIntermediateCase} then finishes the job.
Bounding the sum trivially gives an unconditional bound that is weaker by a factor $p^j$.

Now suppose $p|m_1 m_2 m_3$ and $r=1$.  
We claim that $\widehat{H} = 0$ unless $p^{k-j}|| m_i$, for each $i=1,2,3$.  By symmetry, we may assume $p|m_2$, say. 
Under this condition, the summand in \eqref{eq:HhatFormulaIntermediateCaseAfterGaussSumCalculated} vanishes unless $(p,t) = 1$ in which case we must assume $p^{k-j} || m_3$.  By symmetry again, this implies that the sum vanishes unless $p^{k-j}||m_1, m_2$ also.  
Then $\widehat{H}(\psi,\chi,p^{k-j},p^{k-j},p^{k-j},1)$ equals
\begin{equation*}
 \widehat{H} =  
 p^{k-j} \tau(\overline{\psi})  \sum_{t,u\shortmod{p^k}}
 \chi(t+ p^{k-j} u) \overline{\chi}( t + p^{2(k-j)})
\overline{\chi}(u) \chi(-p^{k-j} +  u) \psi(t)
.
\end{equation*}
Changing variables $t \rightarrow ut$, followed by $t \rightarrow t^{-1}$ and $u \rightarrow u^{-1}$, this becomes
\begin{equation*}
 \widehat{H} =  
 p^{k-j} \tau(\overline{\psi})  \sum_{t,u\shortmod{p^k}}
 \chi(1+ p^{k-j} t) \chi(1-p^{k-j}u) \overline{\chi}( 1 + p^{2(k-j)}tu )
  \overline{\psi}(tu)
.
\end{equation*}
The summand is periodic modulo $p^j$, so it is the same sum repeated $p^{2(k-j)}$ times, and the conjectured bound \eqref{eq:characterSumBoundIntermediateCase} finishes the bound in this case.  Bounding the sum trivially gives an unconditional bound that is weaker by a factor $p^j$.

Lastly, the case with $p|r$ and $p|m_1m_2m_3$ is covered by \eqref{eq:Hhatchivanishesunlessm1m2m3randqarecoprime}.
\end{proof}

The most important case in the evaluation of $\widehat{H}$ occurs with \eqref{eq:Hhatintermsofgchipsi}, and it is crucial to have a strong bound on $g(\chi,\psi)$, which we claim with the following
\begin{mytheo}
\label{thm:gbound}
Let $g(\chi,\psi)$ be given by \eqref{eq:gdef}, where $\chi$ is primitive modulo $q$.  For $q=p$ or $q=p^2$, we have
\begin{equation*}
 |g(\chi, \psi)| \ll q.
\end{equation*}
\end{mytheo}
We prove Theorem \ref{thm:gbound} in Section \ref{section:gbound}.

\subsection{Estimates for $H_{\chi}(m_1, m_2, m_3, r)$ in case some $m_j = 0$.}
The calculations in this section may also be used to bound $H_{\chi}$ in case some $m_j = 0$, by way of \eqref{eq:HchiFourierInversion} (of course, one could calculate $H_{\chi}$ directly).   From Lemma \ref{lemma:HhatEvaluationpsiPrimitive} and the unconditional parts of Lemma \ref{lemma:HhatEvaluationq=pSquaredPsiConductorp}, observe that $\widehat{H}(\psi, \chi, m_1, m_2, m_3, r) = 0$ if some $m_j =0$, except in the case that $\psi$ is the trivial character modulo $q$, in which case from Lemma \ref{lemma:HhatEvaluationPsiTrivial}  we deduce $|\widehat{H}(\psi, \chi, m_1, m_2, m_3, r)| \leq (m_1, q) (m_2, q)(m_3, q)$ by the trivial bound on the Ramanujan sums. 
Therefore by \eqref{eq:HchiFourierInversion}, we have
\begin{equation}
\label{eq:HchiboundSomemjZero}
 |H_{\chi}(m_1, m_2, m_3, r)| \ll q^{-1} (m_1, q)(m_2, q)(m_3, q) q^{\varepsilon}, \quad \text{if } m_1 m_2 m_3 = 0.
\end{equation}
It is useful to record that from \eqref{eq:GHchirelation}, we deduce
\begin{equation}
\label{eq:GtrivialBound}
|G(m_1, m_2, m_3, c)| \ll \frac{q^{\varepsilon}}{cq} \frac{(m_1, q)(m_2, q)(m_3, q)}{q},
\qquad
\text{if } m_1 m_2 m_3 = 0.
\end{equation}

\section{Estimation of $Z_{\text{fin}}$}
Let $\eta_j$, $j=1,2,3,4$ denote any unimodular completely multiplicative functions, and define
\begin{equation*}
Z_{\text{fin},p}(\sigma_1, \sigma_2, \sigma_3, \sigma_4) = \sum_{\substack{a_1,a_2,a_3,d | p^{\infty} \\ (a_1,d) = 1}}
 \frac{\eta_1(a_1) \eta_2(a_2) \eta_3(a_3) \eta_4(d) }{a_1^{\sigma_1} a_2^{\sigma_2} a_3^{\sigma_3} d^{\sigma_4} } \widehat{H}(\psi_p, \chi_p, a_1,a_2,a_3,d).
\end{equation*}

\begin{mylemma}
Let $Z_{{\rm fin}, p}$ be as above, with $q=p^k$, $k \geq 1$, and $\chi_p$ primitive modulo $q$.  
If $\psi_p$ has conductor $p^j$ with $1 \leq j < k$,
assume Conjecture \ref{conj:characterSumBoundIntermediateCase} holds for $\chi_p$, $\psi_p$.
If $\sigma_j \geq \sigma > 1/2$ for all $j$, then
\begin{equation}
\label{eq:ZfinHalfLineBound}
Z_{{\rm fin},p}(\sigma_1, \sigma_2, \sigma_3, \sigma_4) \ll_{\sigma, \eps} \delta_{\psi} q^{1/2} |g(\chi, \psi)| +  q^{3/2+\varepsilon},
\end{equation}
where $\delta_{\psi}$ is the indicator function of the property that $\psi$ is primitive (of conductor $p^k$).
If $\sigma_j \geq \sigma > 1$ for all $j$, and $\psi_p$ is the trivial character, then
\begin{equation}
\label{eq:ZfinOneLineBound}
Z_{{\rm fin},p}(\sigma_1, \sigma_2, \sigma_3, \sigma_4) \ll_{\sigma,  \eps} q^{1+\varepsilon}.
\end{equation}
\end{mylemma}
Remark.  This result is unconditional for $k \leq 2$.

\begin{proof}
First suppose that $\psi$ is primitive modulo $q$. By Lemma \ref{lemma:HhatEvaluationpsiPrimitive}, all terms except $a_1=a_2=a_3=d=1$ vanish, giving the result.

Now suppose that $\psi$ is the trivial character.  By Lemma \ref{lemma:HhatEvaluationPsiTrivial}, we have
\begin{equation*}
|Z_{\text{fin},p}| \leq 
q \sum_{r=0}^{\infty} \frac{(p^k, p^r)}{p^{r \sigma_4}}
+ 
\sum_{a_1, a_2, a_3 \geq 0} 
 \frac{(p^k, p^{a_1}) (p^k, p^{a_2})(p^k, p^{a_3})}{p^{a_1 \sigma_1 +a_2 \sigma_2 + a_3 \sigma_3}}.
\end{equation*} 
which is bounded consistently with the lemma.

Finally, consider $\psi$ of conductor $p^j$, $1 \leq j < k$.   Lemma \ref{lemma:HhatEvaluationq=pSquaredPsiConductorp}, which depends on Conjecture \ref{conj:characterSumBoundIntermediateCase}, gives
 \begin{equation*}
|Z_{\text{fin},p}| 
\ll
\frac{p^{2k-\frac{j}{2}}}{p^{(k-j)\sigma_4}}
+ \frac{p^{3k-\frac{3j}{2}}}{p^{(k-j)(\sigma_1+\sigma_2+\sigma_3)}} \ll p^{3k/2}.
\end{equation*}
This is consistent with \eqref{eq:ZfinHalfLineBound} (note the bound \eqref{eq:ZfinOneLineBound} is not claimed in this case).
\end{proof}

\section{Estimation of $Z$} 
\subsection{The main lemma}
Recall $Z$ is given by \eqref{eq:ZintermsofDirichletLandZfin}.
\begin{mylemma}
\label{lemma:Zproperties}
Suppose $q$ is cube-free.
There exists a decomposition $Z = Z_0 + Z_1$, where $Z_0$ and $Z_1$ satisfy the following properties.  Firstly,
$Z_0$ is meromorphic for $\text{Re}(s_j) \geq \sigma > 1/2$ for all $j$ and
analytic for $\text{Re}(s_j) \geq \sigma >  1$ for all $j$.  It has a pole whenever some $s_j=1$ and the other variables are fixed.  In the region $\text{Re}(s_j) \geq \sigma >  1$ it satisfies the bound
\begin{equation*}
Z_0(s_1, s_2, s_3, s_4) \ll_{\sigma,  \eps} q^{\varepsilon}.
\end{equation*}

Secondly, $Z_1$ is analytic for $\real(s_j) \geq \sigma \geq 1/2$ for all $j$, wherein it satisfies the bound 
\begin{equation}\label{eq:LSlike}
\int_{-T}^{T} |Z_1(\sigma+it, \sigma+it, \sigma+it, \sigma-it)| dt \ll_{\eps} q^{3/2+\varepsilon} T^{1+\varepsilon},
\end{equation}
for $T \gg 1$.  
The same bound stated for $Z_1$ also hold for $Z_0$, provided $1/2 \leq \real(s_j) \leq 0.99$.
\end{mylemma}
Remark.  The statement of Lemma \ref{lemma:Zproperties} is essentially equivalent to \cite[Prop.\ 3]{PetrowYoung}.

\begin{proof}
Let $Z_0$ be the contribution to $Z$ from the trivial character, and let $Z_1 = Z - Z_0$.  All the desired estimates
follow from the previous estimates on $Z_{\text{fin}}$ and a bound on the fourth moment of Dirichlet $L$-functions (see \cite[Lem.\ 8]{Petrow} for instance).
\end{proof}

\begin{myconj}
\label{conj:Zproperties}
 The statement of Lemma \ref{lemma:Zproperties} holds  for any $q$.
\end{myconj}
Remark.  The proofs of the cubic moment bounds only need the properties of $Z$ presented in Lemma \ref{lemma:Zproperties}.  Therefore, if Conjecture \ref{conj:Zproperties} is true, then all the cubic moment bounds stated in the introduction of this paper are valid for arbitrary $q$.

\section{Bounding $g(\chi, \psi)$: the proof of Theorem \ref{thm:gbound}}
\label{section:gbound}

\subsection{The case $q=p$}
\label{section:gboundp}
In this subsection, we prove Theorem \ref{thm:gbound} in the case where $q=p$ is prime. Conrey and Iwaniec \cite{CI} proved $g(\chi,\psi)\ll p$ in the case that $\chi $ is the quadratic character. However, their proof does not seem to generalize: they conclude from Deligne's theorem that the bound $g(\chi,\psi) \ll p$ holds for all except at most one primitive $\psi$. The possible exceptional $\psi$ can only be the quadratic character $\psi = \chi$, and then $g(\chi,\chi)$ has a special structure which Conrey and Iwaniec exploited to show $g(\chi, \chi)\ll p$ by elementary means. When $\chi$ is not quadratic, this special structure is not present, and it is not clear whether the bound $g(\chi,\psi) \ll p$ for $\psi$ quadratic has an elementary proof. 

To prove Theorem \ref{thm:gbound} we instead use Deligne's second proof of the Riemann Hypothesis \cite{DeW2}. We analyze the sum $g(\chi,\psi)$ by writing it as $$ \sum_{ x } \chi(x) \overline{\chi}(x+1) \left( \sum_y \overline{\chi}(y) \chi(y+1) \psi(xy-1)\right),$$ and realize the inner sum as a trace function in $x$ of a sheaf $\mathcal{G}$. The sheaf $\mathcal{G}$ will then be compared with the sheaf that corresponds to the trace function $x \mapsto   \chi(x) \overline{\chi}(x+1)$ to show cancellation in both variables. In executing this strategy, we have benefited greatly from the recent works of Fouvry, Kowalski and Michel, which have served to make the theorems of Deligne and Katz on trace functions more amenable to analytic applications. 

\begin{proof} Suppose that $\chi$ and $\psi$ are primitive modulo $p$, and let $\chi_m,\psi_m$ be the characters derived from $\chi,\psi$ by composing with the norm map $N: \F_{p^m} \to \F_p.$ Let 
\begin{equation*} g(\chi_m,\psi_m) = \sum_{x,y \in \F_{p^m}} \chi_m(x) \overline{\chi_m}(x+1)\overline{\chi_m}(y) \chi_m(y+1) \psi_m(xy-1).\end{equation*}
By the Grothendieck-Lefschetz trace formula \cite[Rapport, Thm.\ 3.2]{SGA45} and the Riemann hypothesis of Deligne \cite{DeW2}, we have that there exist algebraic numbers $\alpha_{i,+}$ and $\alpha_{i,-}$ with $|\alpha_{i,+}|=p^{k_i/2}$, $|\alpha_{i,-}|=p^{\ell_i/2}$ with $k_i,\ell_i \in \Z$ such that \begin{equation}\label{eq:RH2dim}g(\chi_m,\psi_m) = - \sum_{i=1}^{N_+} \alpha^m_{i,+} +   \sum_{i=1}^{N_-} \alpha^m_{i,-}.\end{equation}  Results of Adolphson-Sperber or Katz \cite[Thm.\ 12]{KatzBettiNos} show that $N_+, N_- \ll 1$, independently of $\chi,\psi,p$.  Thus, to prove Theorem \ref{thm:gbound} in the case that $\chi,\psi$ are primitive modulo $p$, it suffices to show that $|\alpha_{i,+}|, |\alpha_{i,-}| \leq p.$  

We show that $|\alpha_{i,+}|, |\alpha_{i,-}| \leq p$ using the theory of $\ell$-adic sheaves and trace functions (for background see \cite{SGA45} \cite{KatzGaussSums} \cite{MichelAppliedCoh}). Let $\ell$ be a prime distinct from $p$ and let $\iota: \overline{\Q}_\ell \to \C$ be a fixed isomorphism. If $X$ is an algebraic variety over $\F_p$ then by ``sheaf'' or ``$\ell$-adic sheaf'' we will mean a constructible $\overline{\Q}_\ell$-sheaf on $X$. Note $\ell$ is always assumed distinct from the base field of $X$. If $\mathcal{F}$ is a sheaf on $X$ and $\overline{x} \in X(\overline{\F}_p)$ is a geometric point of $X$, then we write $\mathcal{F}_{\overline{x}}$ for the stalk of $\mathcal{F}$ at $\overline{x}$.  

For any $\ell$-adic sheaf $\cF$ on $X$, its trace function $t_\cF(x)$ is defined to be the value at $x \in X(\F_p)$ of the trace of the Frobenius endomorphism of $\F_p$ acting on $\mathcal{F}_x$. That is $$t_\cF(x) = \iota((\Tr \cF)(\F_p,x)) = \iota(\Tr(\Fr_p | \cF_x)).$$

Let $$ \cF_1 = \cL_{\chi((Y+1)Y^{-1})}$$ be the Kummer sheaf attached to the character $\chi\left( \frac{Y+1}{Y}\right),$ whose trace function is $\overline{\chi}(y)\chi(y+1)$. Thus $\cF_1$ is an $\ell$-adic sheaf on $\mathbb{A}^1$; it is a middle-extension sheaf, pure of weight 0, lisse on $\mathbb{A}^1 - \{0,-1\}$. It is of rank 1, hence geometrically irreducible.
Let $$\cK = \cL_{\psi(XY-1)}$$ be the middle extension of the Kummer sheaf attached to $\psi(XY-1)$ on $\A^1 \times \A^1$.   Let $Z \subset \mathbb{A}^2$ be the closed set defined by the equation $XY = 1$. The sheaf $\cK$ is lisse, of rank 1 and pure of weight 0 on the dense open set $V = \mathbb{A}^2 - Z$.  Since $\psi$ is non-trivial, the middle extension sheaf is identical to the extension by $0$ of $\cK$ restricted to $V$. 

Let $p_j: \A^1 \times \A^1 \to \A^1$, $j=1,2$ be the two canonical projections, and let $$\cH = p_2^*\cF_1 \otimes \cK.$$ The sheaf $\cH$ is lisse on the dense open set $U = \mathbb{A}^2 - D,$ where $D \subset \mathbb{A}^2$ is the divisor
$$D = Z \cup \mathbb{A}^1 \times \{0\} \cup \mathbb{A}^1 \times \{-1\}.$$ 

For $i=0,1,2$ define the $\ell$-adic sheaves $$T^i_\cK (\cF_1) := R^ip_{1,!}(\cH ),$$
where $R^ip_{1,!}$ is the higher direct image with compact supports.  The sheaf $\cG = T^1_\cK (\cF_1)$ is is the ``cohomological transform'' of $\cF_1$ defined by the ``kernel'' $\cK$, in the sense of Katz's affine cohomological transforms and of \cite{FKMcoh}.
\begin{lem}\label{GLtf}
If $\chi$ and $\psi$ are non-trivial Dirichlet characters modulo $p$, then $$t_{\cG}(x) = - \sum_{y \in \F_p} t_{\cF_1}(y) t_\cK(x,y).$$
\end{lem}
\begin{proof}
Let $U_{x}  = U \cap  \{x\} \times \A^1 $ be the open set on which $\cH$ restricted to $\{x\} \times \A^1$ is lisse. Precisely, we have $U_{x}= \{x\} \times (\A^1 -\{0,-1,1/x\})$. (Below we take restrictions of $p_2^*\cF_1$ and $\cK$ to $\{x\} \times \A^1$ without mention.) 

There are three representations of $\Gal(\Fbar_p/\F_p)$ given by $H_c^i(U_{x,\Fbar_{p}},\cH)$ for $i=0,1,2$. The Grothendieck-Lefschetz trace formula asserts that $$ \sum_{y \in U_{x}(\F_p)} t_{\cH}(x, y) = \Tr(\Fr_p | H_c^0(U_{x,\Fbar_p},\cH))- \Tr(\Fr_p | H_c^1(U_{x,\Fbar_p},\cH))+\Tr(\Fr_p | H_c^2(U_{x,\Fbar_p},\cH)),$$ where $\Fr_p \in \Gal(\Fbar_p/\F_p)$ is the Frobenius automorphism. By standard operations with Galois representations, and the fact that $\cF_1$ and $\cK$ are extension by 0 sheaves, we have that $$ \sum_{y \in \F_p} t_\cF(y) t_\cK(x,y) = \sum_{y \in U_{x}(\F_p)} t_{\cH}(x,y).$$ Furthermore, by the proper base change theorem (see \cite[Arcata, IV, Thm.\ 5.4]{SGA45}) we have that $H_c^i(U_{x,\Fbar_{p}},\cH)$ is naturally isomorphic to the stalk at $x$ of $T^i_\cK(\cF_1)$. Therefore, to prove the lemma, it suffices to show that all of the stalks of $T^0_\cK(\cF_1)$ and $T^2_\cK(\cF_1)$ are 0.

First we show that the stalks of $T^0_\cK(\cF_1)$ are all zero. We claim that $H^0_c(U_{x,\Fbar_p}, \cH)=0$, and so $T^{0}_\cK(\cF_1)=0$ as well. Since $\chi$ is non-trivial, $\cF_1$ is a middle extension sheaf, and so is $p_2^*\cF_1$. Since both $p_2^*\cF_1$ and $\cK$ are middle-extension, we have by e.g.\ \cite[Lem.\ 4.2]{FKMcoh} that $H^0_c(\{x\} \times \A^1 , \cH)=0$.  Let $\pi:(\{x\} \times \A^1) - U_{x}  \to \Spec \overline{\F}_p$ be the structure morphism. The sheaf $R^{-1}\pi_! \cH$ vanishes by definition, so $H^{-1}_c((\{x\} \times \A^1) - U_{x}, \cH)=0$. 
By excision (see \cite[Sommes Trig. $(2.5.1)^*$]{SGA45}) and the vanishing of the above two cohomology groups, we have that $H^0_c(U_{x,\Fbar_p}, \cH)=0$ as well.

Now we show that the stalks of $T^2_{\cK}(\cF_1)$ are all zero. If $\cL_1$ and $\cL_2$ are any two geometrically irreducible sheaves, lisse on $U_{x}$, then $H_c^2(U_{x,\Fbar_p},\cL_1 \otimes \cL_2)\neq 0$ if and only if $\cL_1 \simeq D(\cL_2)$ on a dense open set where both sheaves are lisse, as one can see by the co-invariants formula (see  \cite[(1.4.1)b]{DeW2}) and Schur's Lemma. In our case, it suffices to consider the $G^\text{geom}=\Gal(\overline{\F_p(X)}/\Fbar_p(X))$-invariants acting on the stalk of $p_2^*\cF_1$ and $\cK$ at a lisse geometric point. Since $\chi$ is non-trivial (this is crucial), we have that $\cF_1$ is ramified at $0$ whereas $\cK$ is not. Therefore the inertia group at zero $I_0 \subset G^\text{geom}$ acts non-trivially on the stalk of $\cF_1$ at any lisse point, whereas $I_0$ acts trivially on any stalk of $\cK$. Therefore the two sheaves cannot be geometrically isomorphic, and so the $H_c^2$ vanishes.
\end{proof}
By Lemma \ref{GLtf} and the fact that $\mathcal{F}_1$ is middle extension we have 
\begin{equation}
g(\chi,\psi) = - \sum_{u \in \F_p} \overline{t_{\mathcal{F}_1}(u)}t_\cG(u).
\end{equation}
By the Riemann hypothesis of Deligne \cite[Thm.\ 3.3.1]{DeW2}, $\cG$ is mixed of weights $\leq 1$, so to apply the orthogonality form of the Riemann hypothesis (e.g.\ \cite[Thm.\ 5.2]{MichelAppliedCoh}), we would need to show that the part of weight 1 of $\cG$, say $\cG_0$, is geometrically irreducible and not geometrically isomorphic to $\cF_1$. It is not difficult to see that $\cG$ has generic rank 2, and we would like to argue that this prevents $\cG_0$ from being geometrically isomorphic to $\cF_1$. However, it is less clear that $\cG_0$ itself has generic rank $2$. 

Recall $U = \A^2-D$, and let $j$ be the open embedding of $U$ in $\mathbb{A}^2$. To handle the issue raised in the previous paragraph, let us introduce the modified sheaf 
$$\widetilde{\cH} := j_!(j^* \cH)$$
and the corresponding cohomological transform sheaves $$\widetilde{T}^i_{\cK}(\cF_1) := R^ip_{1,!}(\widetilde{\cH}).$$
We have defined $\widetilde{\cH}$ in order that it satisfy the hypotheses of Deligne's semicontinuity theorem \cite{Laumon}, used in the proof of part (1) of the following Lemma. 
\begin{lem}\label{lem:Gfacts}
Suppose that $\chi$ and $\psi$ are non-trivial modulo $p$. The sheaf $\widetilde{\cG} = \widetilde{T}^1_{\cK}(\cF_1)$
\begin{enumerate}
\item is lisse on the dense open set $W = \A^1 - \{0,-1\}$,
\item is geometrically irreducible and pure of weight $1$ on $W$, and 
\item has generic rank $2$.
\end{enumerate}
\end{lem}
\begin{proof}
\begin{enumerate}
\item   In order to prove that $\widetilde{\cG}$ is lisse on $W$, we will use Deligne's semicontinuity theorem \cite[Cor.\ 2.1.2]{Laumon}. Consider $$\widetilde{p}_1 : \A^1 \times \mathbb{P}^1 \to \A^1,$$ which is a smooth and proper morphism of relative dimension $1$. Abusing notation, we continue to write $\widetilde{\cH}$ for the extension by $0$ of $\widetilde{\cH}$ from $\A^1 \times \A^1$ to $\A^1 \times \mathbb{P}^1$. By definition, we have $\widetilde{\cG} = R^1_{\widetilde{p}_1,*} \widetilde{\cH}$. 

Let $\widetilde{D}$ be the complement in $\A^1 \times \mathbb{P}^1$ of the open set $U$, that is $\widetilde{D} = D \cup \A^1 \times \{\infty\}$. By restriction, $\widetilde{p}_1$ defines a proper smooth morphism of relative dimension $1$ $$ X \to W = \A^1 -\{0,-1\},$$ where $X = \widetilde{p}_1^{-1}(W)$. The intersection $\widetilde{D} \cap X$ is a divisor in $X$, which is flat and finite (of degree 4) over $W$. The sheaf $\widetilde{\cG}$ is lisse on the complement of $\widetilde{D}\cap X$ in $\A^1 \times \mathbb{P}^1$. 

Let $x \in W$. We identify the fiber $C_x$ of $\widetilde{p}_1$ over $x$ with $\mathbb{P}^1$, and so the restriction of $\widetilde{\cH}$ to $C_x$ is identified with a lisse sheaf on the dense open set (abusing notation) $$U_x = \A^1 - \{0,-1,1/x,\infty\} \subset \mathbb{P}^1.$$ The restriction of the sheaf $\widetilde{\cH}$ to $C_x$ is at most tamely unramified everywhere, hence the function $\varphi$ of \cite[Thm.\ 2.1.1]{Laumon} is constant equal to $0$ on points of $W$. Then we have by Corollary 2.1.2 of loc.\ cit.\ that $R^1\widetilde{p}_{1,*}\widetilde{\cH}$ is lisse on $W$. 
\item The sheaf $\widetilde{\cG}$ is mixed of weights $\leq 1$ on $W$ by the Riemann hypothesis of Deligne \cite[Thm.\ 3.3.1]{DeW2}. Furthermore, the part of $\widetilde{\cG}$ of weight $1$ is geometrically irreducible on $W$ by the diophantine criterion for geometric irreducibility. Indeed, $\cF_1$ is not geometrically isomorphic to the Kummer sheaf $\cL_\psi$ attached to $\psi$, since $\cF_1$ is ramified at $-1$ and $\cL_\psi$ is not, and so the hypothesis of \cite[Prop.\ 5.9(2)]{FKMcoh} is satisfied. 

Finally, we prove that $\widetilde{\cG}$ is pure of weight 1 on $W$ by applying \cite[Lem.\ 4.22(2)]{KMS} to the morphism $\widetilde{p}_1:X \to W$ and the sheaf $\widetilde{\cH}$ on $X$. For any $x \in \mathbb{P}^1$, the sheaf pullback $\widetilde{\cH}_x$ to $C_x = \{x\} \times \mathbb{P}^1$ has no punctual section nor trivial subrepresentation (as a lisse sheaf on $U_x$). Thus the assumptions of loc.\ cit.\ are satisfied. 

It follows that for any $x \in W$, the part of weight $<1$ of the stalk at $x$ of $\widetilde{\cG} = R^1 \widetilde{p}_{1,*} \widetilde{\cH}$ is isomorphic to $$\bigoplus_{y \in C_x - U_x} ( \widetilde{\cH}_x)^{I_y}_{\overline{\eta}} / ( \widetilde{\cH}_x)_{\overline{y}}.$$ But we already have $ ( \widetilde{\cH}_x)^{I_y}_{\overline{\eta}}= 0$ at all singular points $y \in \{0,-1,1/x,\infty\},$ so the direct sum vanishes. From Deligne's Theorem, we conclude that $\widetilde{\cG}$ is pure of weight 1 on $W$. 
\item The stalk of $\widetilde{\cG}$ over $x\in \F_p$ is $H^1_c(\{x\} \times \A^1, \widetilde{\cH}).$ By the Euler-Poincar\'e formula  \cite[8.5.2, 8.5.3]{KatzGaussSums}, if $x \neq -1$ then the dimension of this cohomology group is $-1+3=2$ for the 3 tamely ramified points $0,-1,1/x$ of $\widetilde{\cH}$.  Hence the generic rank is $2$.
\end{enumerate}
\end{proof}
To compare $\widetilde{\cH}$ with $\cH$, observe that the stalks of $\widetilde{\cH}$ are equal to those of $\cH$ outside $D$, and are $0$ along $D$. Thus, the stalks of $\cH$ and $\widetilde{\cH}$ may only differ on $\widetilde{D}$ (and this can only happen if $\chi = \overline{\psi}$). Indeed, in the case $\chi= \overline{\psi}$ note that $t_\cH(-1,-1) = \overline{\chi}(-1)^2$, whereas $t_{\widetilde{\cH}}(-1,-1) = 0$. As an aside, one wonders whether $g(\chi,\overline{\chi})$ admits a ``special structure'' for $\chi$ complex that could be used to give a simpler the proof of the bound $g(\chi,\overline{\chi}) \ll p$ in that special case, as was exploited by Conrey and Iwaniec in the case that $\chi$ is quadratic \cite[\S 14]{CI}. 

By the discussion in the preceding paragraph, we have that $t_\cG(u) = t_{\widetilde{\cG}}(u) + O(1)$. Thus, setting  $$ s(\chi,\psi) = - \sum_{\F_p-\{0,-1\}} \chi(x) \overline{\chi}(x+1) \sum_{y} \overline{\chi}(y)\chi(y+1)\psi(xy-1),$$ we have 
 \begin{equation}\label{eq:gtos} g(\chi,\psi) =  -s(\chi,\psi) + O(p)  = - \sum_{x \in W(\F_p )} \overline{t_{\cF_1}(x)} t_{\widetilde{\cG}}(x)+ O(p).\end{equation} 

Since $\cF_1$ is pure of weight 0 and rank 1 on $W$ and $\widetilde{\cG}$ is pure of weight 1 and geometrically irreducible of rank $2$ on $W$ by Lemma \ref{lem:Gfacts}, $\widetilde{\cG}$
 cannot be isomorphic to $\cF_1$. 
Thus, the Grothendieck-Lefschetz trace formula, the co-invariants formula \cite[(1.4.1)b]{DeW2}, and the Riemann hypothesis of Deligne imply that there exist algebraic numbers $\beta_{i,+}$ and $\beta_{i,-}$, with $|\beta_{i,+}|\leq p$, $|\beta_{i,-}|\leq p$  such that \begin{equation}\label{eq:RH1dim}s(\chi_m,\psi_m) = \sum_{x \in W(\F_{p^m})} \overline{t_{\cF_1}(N(x))} t_{\widetilde{\cG}}(N(x)) = - \sum_{i=1}^{M_+} \beta^m_{i,+} +   \sum_{i=1}^{M_-} \beta^m_{i,-}.\end{equation} 

Here it is \emph{not} clear that $M_+$ and $M_-$ are bounded independently of $\chi, \psi, p$. However, we can avoid this issue by appealing to the two-dimensional Riemann hypothesis of Deligne \eqref{eq:RH2dim}, in which situation we know that $N_+,N_-\ll 1$. 
A slight variation of \cite[Lem.\ 13.2]{CI} shows that $|\alpha_{i,+}|, |\alpha_{i,-}| \leq p^{3/2},$ and we would like to show in fact that $\alpha_{i,+}$ and $\alpha_{i,i}$ are bounded by $p$. Suppose not. Then we would have $$ \limsup_{m \to \infty} \frac{ |g(\chi_m,\psi_m)|}{p^{3m/2}} >0.$$ But this is impossible by \eqref{eq:RH1dim} since $|\beta_{i,+}|$, $|\beta_{i,-}| \leq p$. Therefore $|\alpha_{i,+}|, |\alpha_{i,-}| \leq p,$ so by \eqref{eq:RH2dim} and the fact that $N_+,N_-\ll 1$ we have $g(\chi,\psi) \ll p$ for all $\chi,\psi$ primitive. 

If $\psi$ is not primitive, it must be the trivial character $\psi_0$, in which case we have $g(\chi,\psi_0) \ll p$ by Lemma \ref{lemma:LFTprimeSquared}, which completes the proof of Theorem \ref{thm:gbound} when $q=p$. \end{proof}
Remark. A proof of a minor variant of Theorem \ref{thm:gbound} also appears as \cite[Thm.\ 5.7]{FKMcoh}, from which we drew inspiration in giving the proof that appears above. However, our proof departs from that of loc.\ cit.\ in that we have completely avoided the difficult main Theorems 2.3, 2.5, and 5.8 of \cite{FKMcoh} on the behavior of the conductor under cohomological transforms.

\subsection{The case $q=p^2$}
\label{section:gboundpsquared}
This case can be treated by elementary means.  Since $\chi$ is a Dirichlet character modulo $p^2$, the function $t \mapsto \chi(1+pt)$ is an additive character on $\mz/p\mz$, so there exists an integer $\ell_{\chi}$ so that
\begin{equation*}
 \chi(1+pt) = e_p(\ell_{\chi} t).
\end{equation*}
Note that $\chi$ has conductor $p^2$ if and only if $(\ell_{\chi}, p) = 1$.
Hence if $a,b$ are integers with $(a,p) = 1$, then
\begin{equation}
\label{eq:chiformulaapluspb}
\chi(a + p^{} b) = \chi(a) \chi(1 + p^{} \overline{a} b)
= \chi(a) e_p(\ell_{\chi} \overline{a} b),
\end{equation}
where $\overline{a} \in \Z$ satisfies $a \overline{a} \equiv 1 \pmod{p^2}$.
Now, for each $t,u \pmod q$ choose $a,b,c,d \in \Z$ such that $a + p^{} b \equiv t \pmod q$ and $c + p^{} d \equiv u \pmod q$. 
We have 
\begin{equation*}
\psi(ut-1) = \psi(-1 + a  c + p^{ }(bc + ad))
= \psi(-1 + a   c) e_p(\ell_{\psi}(bc+ad) \overline{(-1+ac)}).
\end{equation*}
Note that as $t,u$ run through $\Z/q\Z$, each of the integers $a,b,c,d$ represent every residue class modulo $p$.  We obtain
\begin{multline}
\label{eq:gformula1}
g(\chi, \psi)
= \sum_{a, c}
\chi(a) \overline{\chi}(a+1)
\overline{\chi}(c) \chi(c+1)
\psi(-1 + ac)
\\
 \sum_{b,d }
 e_p(\ell_{\chi} b( \overline{a} - \overline{(a+1)}) + \ell_{\psi} bc \overline{(-1+ac)})
 e_p(-\ell_{\chi} d( \overline{c} - \overline{(c+1)}) + \ell_{\psi}ad \overline{(-1+ac)}).
\end{multline}
In particular, we have $(a(a+1) c(c+1)(ac-1), p) = 1$ for every non-zero  term of the sum over $a$ and $c$, so all inversions modulo $p$ here and below are justified. 
The sum over $b$ equals $p$ provided
\begin{equation}\label{Lcong1}
\ell_{\chi}(\overline{a} - \overline{(a+1)}) \equiv - \ell_{\psi} c \overline{(-1+ac)} \pmod{p},
\end{equation}
and vanishes otherwise.  Similarly, 
the sum over $d$ equals $p$ provided
\begin{equation}\label{Lcong2}
\ell_{\chi}(\overline{c} - \overline{(c+1)}) \equiv  \ell_{\psi} a \overline{(-1+ac)} \pmod{p},
\end{equation}
and vanishes otherwise. 
We claim that there at most $2$ solutions to \eqref{Lcong1} and \eqref{Lcong2}, whence
\begin{equation*}
 |g(\chi, \psi)| \leq 2 q,
\end{equation*}
for $q=p^2$.  Along the way, we will also see that $g(\chi, \psi) = 0$ if $\psi$ is not primitive.

Indeed, multiplying the first congruence by $a(a+1)$ and the second one by $c(c+1)$, we obtain the equivalent system
\begin{equation*}
\ell_{\chi} \equiv - \ell_{\psi} ac(a+1) \overline{(-1+ac)} \pmod{p^{ }}, \qquad \ell_{\chi} \equiv \ell_{\psi} a c(c+1) \overline{(-1+ac)} \pmod{p^{ }}.
\end{equation*}
Since $(\ell_{\chi}, p) = 1$, this implies that $g(\chi, \psi) = 0$ unless $(\ell_{\psi}, p) = 1$, which means $\psi$ is primitive.
Furthermore, we deduce that $a(a+1) c \equiv - ac(c+1) \pmod{p^{ }}$, whence 
$c \equiv -2 - a \pmod{p^{ }}$, which uniquely determines $c$ in terms of $a$.  Then we see that $a$ must satisfy
\begin{equation}\label{eq:a}
 a(a+1)(a+2) \overline{(-1+a(-2-a))} \equiv  \overline{\ell_{\psi}} \ell_{\chi} \pmod{p^{ }}.
\end{equation}
Setting $A = \overline{\ell_{\psi}} \ell_{\chi}$, \eqref{eq:a} is equivalent to  
\begin{equation*}
 a(a+2) \equiv  -A  (a+1) \pmod{p^{}}. 
\end{equation*}
Hence $a$ satisfies a certain monic quadratic polynomial, having at most $2$ solutions modulo $p$. 
This gives the desired bound on $g$, completing the proof of Theorem \ref{thm:gbound}.

\section{Archimedean aspects, part 1}
\label{section:WeightFunctions1}

In this section, we derive the analytic properties of the weight functions $J^\pm_0$ and $J^\pm$ defined in \eqref{eq:Jpm0def} and \eqref{eq:Jpmdef}.

\subsection{Approximate functional equations}
Recall from \eqref{eq:Vjdef} the functions $V_j(y,t)$.
\begin{mylemma}
\label{lemma:Vproperties}
For each $j=1,2$, $V_j(y,t)$ is an entire, even function in $t$, for any given $y>0$.  Moreover, for $t \in \R$ it satisfies the bound
\begin{equation}
\label{eq:Visalmostinert}
y^k (1/2 + it)^{\ell} \frac{\partial^{k+\ell}}{\partial y^k \partial t^{\ell}} V_j(y,t) \ll_{A,k,\ell} \Big(1+\frac{y}{1+|t|^j}\Big)^{-A},
\end{equation}
for any $A>0$.  
For $t = -i/2 + v$ with $v \in \mr$, we have for any $A>0$
\begin{equation}
\label{eq:VboundShiftedLine}
y^k \frac{\partial^{k}}{\partial y^k} V_j(y,-\tfrac{i}{2} + v) \ll_{A,k} \Big(1+\frac{y}{1+|v|^j}\Big)^{-A}.
\end{equation}
\end{mylemma}
\begin{proof}
By shifting the contour far to the right, we see that $V_j(y,t)$ is analytic for $t$ in any fixed horizontal strip, so it can be extended as an entire function of $t$.  It is clearly invariant under $t \rightarrow -t$.

Now assume $t \in \R$. First we show \eqref{eq:Visalmostinert} in the case $k=\ell=0$.  
We assume $\delta =0$ for notational simplicity, the $\delta =1$ case being similar.
Stirling's asymptotic expansion gives that $\log \Gamma(z) = (z-\frac12) \log(z) - z + \sum_{j=0}^{N} c_j z^{-j} + O(|z|^{-N-1})$, for certain constants $c_j$, valid for $\text{Re}(z) \geq 1/100$, say.  From this we deduce that if $|a|^2 \leq |z|$, $\text{Re}(z) \geq 1/4$, 
then
\begin{equation}
\label{eq:StirlingRatio}
 \log \frac{\Gamma(z+a)}{\Gamma(z)} = a \log{z} + \sum_{j=1}^{N} \frac{P_j(a)}{z^j} + O\Big(\frac{(1+|a|)^{2N+2}}{|z|^{N+1}}\Big),
\end{equation}
for certain polynomials $P_j$ of degree at most $2j$.
Fix $\sigma \in \mr$ so that $1/2   + \sigma > 0$.  
Then for $\real(s) = \sigma$ and $|\text{Im}(s)|  \leq (1+|t|)^{1/4}$, we derive from \eqref{eq:StirlingRatio} that
\begin{multline}
\label{eq:StirlingAsymptoticExpansion}
 \exp(s^2) \frac{\Gamma_{\mr}(1/2+s+it)}{\Gamma_{\mr}(1/2+it)} 
 \\
 = \Big(\frac{|t|}{\pi}\Big)^{s/2} \exp(s^2) \Big(
  1 + \sum_{j=1}^{N} \frac{P_j(s)}{(1/2+it)^j} 
  + O_{\sigma, N}((1+|t|)^{-\frac{N+1}{2}})\Big),
\end{multline}
provided $t$ is sufficiently large, and where $P_j$ is a different collection of polynomials of degree $\leq 2j$.
If $|\text{Im}(s)| > (1+|t|)^{1/4}$, then a crude application of Stirling gives 
\begin{equation*}
 \exp(s^2) \frac{\Gamma_{\mr}(1/2+s+it)}{\Gamma_{\mr}(1/2+it)} 
 =
   O( (1+|t|)^{\sigma/2} \exp(-\text{Im}(s)^2/2)).
\end{equation*}
In any event, we shift the contour to $\real(s) = A$ to see that $V_j(y,t) \ll_{A} (1+|t|^j)^A y^{-A}$ for $y > 1+|t|^j$.  If $y \leq 1+|t|^j$ we instead move the contour to $\sigma = -1/4$, say.  Accounting for the pole and bounding the integral on the new line, we obtain $V_j(y,t) \ll 1$ in this case.

Next we consider derivatives.  Note that differentiation $k$ times with respect to $y$ followed by multiplication by $y^k$ gives an integral of the form \eqref{eq:Vjdef} back, but with $G_j(s)$ multiplied by a polynomial in $s$.  The exponential decay of $G_j(s)$ easily accomodates for this, showing \eqref{eq:Visalmostinert} for $\ell=0$, and any $k \geq 0$.
Differentiation of Stirling's formula with respect to $t$ leads to \eqref{eq:Visalmostinert} for any $k, \ell$.

Next consider the case $t= -i/2 + v$ with $v \in \mr$, so $it = 1/2 + iv$.  For $y > 1 + |t|^j$ we move the contour far to the right and bound it the same way.  For $y \leq 1 + |t|^j$, we shift left, to $-1/4$ again.  We pass poles at $s=0$ (as before) giving a residue of $1$, and at $s=-1/2+iy = iv$.  This latter residue is $O((1+|v|)^{-100})$, i.e. uniformly bounded for $v \in \mr$, using that the apparent pole of $\frac{1}{iv}$ at $v=0$ is cancelled by a zero of $1/\Gamma_{\mr}(-iv)$.  It is not hard to see that \eqref{eq:VboundShiftedLine} holds.  
\end{proof}

\subsection{Properties of $J^+$}
\label{section:Jplus}
We invite the reader to recall the definition of inert functions from Section \ref{section:definitionsandconventions}. 

\begin{mylemma}
\label{lemma:Jpluspropertiesxsmall} 
 Let $J^{+}(x, \cdot)$ be defined as in \eqref{eq:Jpmdef}.  Then
 \begin{equation}
 \label{eq:Jplusbound}
 \frac{\partial^{k} }{\partial x^k } J^{+}(x, \cdot)  \ll_{k} x^{} (x^{-k} + x^{k}) T^{k+1},
 \end{equation}
 and $J^+(x, \cdot)$ is $1$-inert with respect to the variables $t_1, t_2, t_3, c, m_1, m_2, m_3$.
\end{mylemma}
We will use this for the relatively small values $x \ll T^{2+\varepsilon}$. 
In the complementary range, we have the following.
\begin{mylemma}
\label{lemma:JplusAsymptoticExpansion}  
 Suppose for some $\eps>0$ that $1 \leq T^{2+\varepsilon} \ll x$.  Then for any $A>0$
 \begin{equation*}
  J^+(x, \cdot) = \sum_{\pm} T^2 x^{-1/2} e^{\pm ix} W_{\pm}(x, \cdot) + O_{\eps, A}(x^{-A}),
 \end{equation*}
 where $W_{\pm}(x,\cdot)$ is a function (depending additionally on $\varepsilon$, $T$, $A$) satisfying $x^k \frac{\partial^k}{\partial x^k} W_{\pm}(x, \cdot) \ll 1$.  Moreover, $W_{\pm}(x,\cdot)$ is $1$-inert with respect to the variables $t_1, t_2, t_3, m_1, m_2, m_3, c$.
\end{mylemma}

\begin{proof}[Proof of Lemma \ref{lemma:Jpluspropertiesxsmall}]
 First consider the case $k=0$.  In \eqref{eq:Jpmdef} we shift the contour to the line $\imag(t)=-1/2$.  Then from \eqref{eq:VboundShiftedLine}, and using $ |\cosh(-\tfrac{\pi i}{2} + \pi y)| = |\sinh(\pi y)|$, we have
 \begin{equation*}
  |J^+(x, \cdot)| \ll \intR \frac{|J_{1+2iy}(x)|}{|\sinh(\pi y)|} \frac{|y|(1+y^2)}{T^2} \exp(-(y/T)^2) dy.
 \end{equation*}
Next we claim that for any integer $a \geq 0$ we have
\begin{equation}
\label{eq:JBesselBoundTrivial}
  \frac{|J_{1+a+2iy}(x)|}{|\sinh( \pi y)|} \ll \frac{1+|y|}{|y|} \frac{(x/2)^{1+a}}{|1/2+2iy|^{a+1}}.
\end{equation}
 This bound can be derived with a little work from \cite[8.411.4]{GR} and Stirling's approximation.
Taking $a=0$, this implies \eqref{eq:Jplusbound} for $k = 0$.

We next extend this to $k \geq 1$.
 By \cite[8.472.2, 8.486.13]{GR} we have
 \begin{equation}
 \label{eq:JBesselRecursionFormula}
  \frac{d}{dx} Z_{\nu}(x) = \frac{\nu}{x} Z_{\nu}(x) - Z_{\nu+1}(x),
 \end{equation}
 valid for $Z_{\nu} = J_{\nu}$ as well as $Z_{\nu} = K_{\nu}$.
Iterating this, we derive
\begin{equation}
\label{eq:JrecursiveDerivatives}
 \frac{d^k}{dx^k} J_{\nu}(x) = \sum_{m=0}^{k} \frac{P_{k,m}(\nu)}{x^{m}} J_{\nu+k-m}(x),
\end{equation}
where $P_{k,m}$ is a polynomial of degree at most $m$.
By \eqref{eq:JrecursiveDerivatives} and \eqref{eq:JBesselBoundTrivial}, we deduce that
\begin{equation*}
  \frac{|\frac{d^k}{dx^k} J_{1+2iy}(x)|}{|\sinh( \pi y)|}
  \ll_k \frac{1+|y|}{|y|} \sum_{m=0}^{k} 
  \frac{(1+|y|)^{m}}{x^m}
  \frac{x^{1+k-m}}{(1+|y|)^{1+k-m}} 
  \ll \frac{x}{|y|} \Big(\frac{x^k}{(1+|y|)^k} + \frac{(1+|y|)^k}{x^k}\Big).
\end{equation*}
 It is then straightforward to derive \eqref{eq:Jplusbound} for all $k$.
 
 The final statement of the lemma, that $J^+(x,\cdot)$ is $1$-inert with respect to the other variables, follows from Lemma \ref{lemma:Vproperties}, since the only dependence of $J^+$ on these auxiliary parameters is via the factors $V_1(y_1, t) V_2(y_2, t)$ and the inert function $w$.
 \end{proof}

\begin{proof}[Proof of Lemma \ref{lemma:JplusAsymptoticExpansion}]

By \cite[8.411.11]{GR} and an interchange of orders of integration justified by integration by parts and Fubini, there exists an integral representation in the form
 \begin{equation*}
  J^{+}(x, \cdot) =  w(\cdot) T^2 \int_0^{\infty} \cos(x \cosh(v)) g(v, \cdot) dv,
 \end{equation*}
 where 
 \begin{equation*}
 g(v, \cdot) = T^{-2} \intR t \tanh(\pi t) \frac{t^2+\frac14}{T^2} \exp(-(t/T)^2) \cos(2tv)V_1(\cdot,t)V_2(\cdot,t)\, dt.  
 \end{equation*}
 Here $g$ is a Schwartz-class function, more precisely it satisfies the bounds
 \begin{equation}
 \label{eq:gbound}
  g^{(j)}(v, \cdot) \ll_{A,j} T^{j} (1+|v|)^{-A},
  \qquad \text{for all $A>0$, $j \geq 0$},
 \end{equation}
 and is $1$-inert with respect to the other variables by Lemma \ref{lemma:Vproperties}.  
 Hence
 \begin{equation*}
  J^{+}(x, \cdot) = \sum_{\pm }  T^2 \int_0^{\infty} e^{\pm ix \cosh(v)} g(v, \cdot) dv =  \sum_{\pm} T^2 e^{\pm ix} F_{\pm}(x, \cdot),
 \end{equation*}
 where
 \begin{equation*}
  F_{\pm}(x, \cdot) = \int_0^{\infty} e^{\pm ix(\cosh v -1)}   g(v, \cdot) dv.
 \end{equation*}
 It suffices to show that $F_\pm(x,\cdot) = \frac{1}{x^{1/2}} W_\pm(x, \cdot) + O_A(x^{-A})$ with $W_\pm(x, \cdot)$ satisfying the required derivative bounds. For notational simplicity, we consider only the $+$ case, which we write as $F(x, \cdot)$ for $F_{+}(x, \cdot)$.

Write a smooth partition of unity of the form
$$1 = f_0(v/U) + \sum_{V} f_1(v/V) + f_2(v)
\qquad \text{for } v > 0,
$$
with the following properties: $f_0$ has support on $[-1,1]$, $f_1$ has support on $[1,2]$, $f_2$ vanishes on  $[0,1]$, $U=x^{-1/2+\eps}$, and $V$ runs over $O(\log x)$ real numbers with $U \ll V \ll 1$. 
Repeated integration by parts shows that for all sufficiently large $A>0$
\begin{equation}
\label{eq:stapler}
 \int_{1}^{\infty} e^{ix(\cosh v -1)} g(v, \cdot) f_2(v)\,dv \ll T^{ j} x^{-j} \ll_{A} x^{-A},
\end{equation}
taking $j$ large, and using $x \gg T^2$. Similarly, applying \cite[Lem.\ 8.1]{BKY} with parameters $(X,Y,Q,R,U, \alpha, \beta)$ in our situation taking the values $(1, x, 1, xV, 1, V, 2V)$, we see that
$$ \intR e^{ix(\cosh v -1)} g(v, \cdot) f_1(v/V)\,dv \ll_A x^{-A}.$$  Hence 
\begin{equation*}
 F(x, \cdot) = \int_{0}^{2} e^{ix(\cosh v -1)}  g(v, \cdot) f_0\Big(\frac{v}{U}\Big) dv + O_{A}(x^{-A}).
\end{equation*}

Now let us develop $e^{ix(\cosh v -1)}$ by first taking the Taylor expansion for $\cosh v - 1$, and then expanding it in the power series expansion for $\exp$. We get that $$e^{ix(\cosh v -1)} = e^{ixv^2/2} \sum_{m = 0}^M b_m\Big(x\sum_{n = 0}^N a_{n} v^{2n+4}\Big)^{m} + O \left( xv^{2N+6} + (xv^4)^{M+1}\right).$$ Since $v\ll x^{-1/2+\eps}$, we may take $M,N$ large enough depending on $\eps,A$ so that $$e^{ix(\cosh v -1)} = e^{ixv^2/2} \sum_{m,n \geq 0} c_{m,n} (xv^2)^m v^{2n} + O_A(x^{-A}),$$ for some new coefficients $c_{m,n}$ and where all but finitely many of the $c_{m,n}$ are zero.

Thus
\begin{equation*}
F(x, \cdot) = \sum_{m \leq M ,n \leq N} c_{m,n} \int_0^{\infty} (xv^2)^m v^{2n} e^{ix v^2/2}   g(v, \cdot) f_0\Big(\frac{v}{U}\Big) dv + O_{A}(x^{-A}).
\end{equation*}
It transpires that $g$ is nearly constant on the support of $f_0$.  To see this, we note that
\begin{equation*}
 UT \ll x^{-\delta},
\end{equation*}
for some $\eps>\delta > 0$, where $\eps$ is the $\eps$ appearing in $x \gg T^{2+\varepsilon}$, and we have chosen the $\eps$ in the definition of $U$ small enough in comparison. Then, for any $L$ we have \begin{equation*}
 g^{(L)}(\xi) v^{L} \ll   (UT)^{L} \ll  x^{-L \delta},
\end{equation*}
so that we may develop $g$ in a Taylor series around $0$ with finitely many terms, the number of which only depends on $A,\eps$.  
Hence
\begin{equation*}
F(x, \cdot) = \sum_{\ell \leq L, m \leq M, n \leq N} c_{\ell, m,n}   g^{(\ell)}(0) \int_0^{\infty} (xv^2)^m v^{2n+\ell} e^{ix v^2/2}  f_0\Big(\frac{v}{U}\Big) dv + O_{A}(x^{-A}),
\end{equation*}
for all sufficiently large $L,M,N$.
Changing variables $v \rightarrow x^{-1/2} v$, we obtain 
\begin{equation}
\label{eq:WPrettyGoodFormula}
F(x, \cdot) = x^{-1/2} \sum_{\ell,m,n} c_{\ell, m,n}   \frac{g^{(\ell)}(0)}{x^{\ell/2}} x^{-n}  \int_0^{\infty}  v^{2m+2n+\ell} e^{i  v^2/2}  f_0\Big(\frac{v}{x^{\varepsilon}} \Big) dv + O_{A}(x^{-A}).
\end{equation}

Let us analyze the inner integral.  We claim
\begin{equation*}
 \int_0^{\infty}  v^{N} e^{i  v^2/2}  f_0\Big(\frac{v}{x^{\varepsilon}} \Big) dv
 = C(N) + O_{N, A}(x^{-A}),
\end{equation*}
for some constant $C(N)$ independent of $f_0$ and $x$.

\begin{proof}[Proof of claim]
For a smooth function $f$ supported on $|v| \ll 1$, define
\begin{equation*}
 I(N,f,V) = \int_0^{\infty} v^{N} e^{iv^2/2} f(v/V) dv,
\end{equation*}
where $V \gg 1$ is large.
Our first observation is that \cite[Lem.\ 8.1]{BKY} shows that $I(N,f,V) \ll_{A,N,f} V^{-A}$ provided $f$ is supported on $[1/2,4]$, say.  Our $f_0$ is not supported on this interval, but this argument shows $I(0, f_0, V) = I(0,1,V) + O_{A}(V^{-A})$, where $I(0,1,V) = e^{\pi i/4} \sqrt{\frac{\pi}{2}}$.  Next, an integration by parts argument shows
\begin{equation*}
 I(N,f,V) =   i \delta_{N=1} f(0) + i (N-1) I(N-2,f,V) + i V^{-1} I(N-1,f',V).
\end{equation*}
Here we interpet $I(M,f,V) = 0$ if $M < 0$.
Since $f_0'$ is dyadically-supported, this implies
\begin{equation*}
 I(N,f_0, x^{\varepsilon}) =  i \delta_{N=1} + i (N-1) I(N-2, f_0, x^{\varepsilon}) + O_{N,A}(x^{-A}).
\end{equation*}
Repeating, we obtain the claim.
\end{proof}

Applying the claim to \eqref{eq:WPrettyGoodFormula}, we have
\begin{equation*}
 F(x, \cdot) = x^{-1/2} \sum_{\ell \leq L, m \leq M, n \leq N} c_{\ell, m,n} \frac{ g^{(\ell)}(0)}{x^{\ell/2}} x^{-n} + O_{A}(x^{-A}),
\end{equation*}
for some newly re-defined sequence of coefficients $c_{\ell,m,n}$, which completes the proof. 
\end{proof}

\subsection{Properties of $J^-$}
\label{section:Jminus}
\begin{mylemma}
\label{lemma:Jminuspropertiesxsmall}
We have
 \begin{equation}
 \label{eq:Jminusbound}
\frac{\partial^{k} }{\partial x^k } J^{-}(x, \cdot)  \ll_{k, \eps} x^{1-\eps} (x^{-k}+x^{k})  T^{1+k+\eps}.
 \end{equation}
 Moreover, $J^-(x,\cdot)$ is $1$-inert with respect to the variables $t_1, t_2, t_3, c, m_1, m_2, m_3$.
\end{mylemma}
As in the $J^+$ case, this lemma is of interest to us when $x$ is not too large.  In the complementary case we have the following.
\begin{mylemma}
\label{lemma:JminusBound}
 Suppose for some $\eps>0$ that  $1 \leq T^{1+\eps} \ll x  $.  Then $J^{-}(x, \cdot) \ll_{A} x^{-A}$

\end{mylemma}

 \begin{proof}[Proof of Lemma \ref{lemma:Jminuspropertiesxsmall}]
 As in the proof of Lemma \ref{lemma:Jpluspropertiesxsmall}, the property that $J^{-}$ is $1$-inert with respect to the other variables  is easy to see, so we now focus on the bound \eqref{eq:Jminusbound}.
 By \cite[8.486.10]{GR}, we have
 \begin{equation}\label{eq:J-int}
  J^{-}(x, \cdot) = \frac{x}{i \pi^2} \intR (K_{1+2it}(x) - K_{1-2it}(x)) \sinh(\pi t)  \exp(-(t/T)^2) \frac{(t^2+\frac14)}{T^2} V_{1}(\cdot, t) V_{2}(\cdot,t) dt.
 \end{equation}
From \cite[8.432.5]{GR}, that is,
\begin{equation*}
K_{\nu}(x) = \frac{\Gamma(\nu + \frac12) 2^{\nu}}{x^{\nu} \Gamma(\frac12)} \int_0^{\infty} \frac{\cos(xt)}{(t^2 + 1)^{\frac12 + \nu}}  dt, \qquad \text{Re}(\nu) > 0,
\end{equation*}
one may readily deduce that
\begin{equation}
\label{eq:KboundTrivial}
 K_{{\eps}+2iy}(x) \ll_{{\eps}} \frac{(1+|y|)^{{\eps}}}{x^{{\eps}} \cosh( \pi y)}
\end{equation}
for $y \in \mr$.
For the part of the integral \eqref{eq:J-int} with $K_{1+2it}$ we move the contour to $\text{Re}(1+2it) = {\eps} > 0$, in all giving a contribution to $J^{-}(x)$ of size $\ll x^{1-{\eps}} T^{1+{\eps}}$.  A similar bound works for the part of the integral with $K_{1-2it}(x)$, but by shifting the contour the other way.  
This gives the desired bound for $k=0$.

Next we sketch how to treat $k \geq 1$.  The bound on $K_{{\eps}+2iy}$ given in \eqref{eq:KboundTrivial} has the same essential features as \eqref{eq:JBesselBoundTrivial}.  Moreover, the $K$-Bessel function satisfies 
\eqref{eq:JBesselRecursionFormula}. 
The same method used for $J^{+}$ now carries over to $J^{-}$ without any significant changes.
 \end{proof}

\begin{proof}[Proof of Lemma \ref{lemma:JminusBound}]
 From \cite[8.432.4]{GR}  one may derive 
\begin{equation*}
  J^{-}(x, \cdot) = T^2 \intR \cos(x \sinh(v)) g(v, \cdot) dv,
 \end{equation*}
 where $g$ satisfies \eqref{eq:gbound}.  (Here $g(v, \cdot)$ may differ slightly from that occuring in the proof of Lemma \ref{lemma:JplusAsymptoticExpansion}, but only by an absolute constant).
 
 As in the proof of Lemma \ref{lemma:JplusAsymptoticExpansion}, we can cut the integral at $|v| \leq 1$ again (with a smooth cutoff), since repeated integration by parts shows the complement is $O_{A}(x^{-A})$ for any $A>0$.
 Therefore,
\begin{equation*}
J^{-}(x, \cdot) = T^2 \intR \cos(x \sinh v) g_1(v, \cdot)  dv + O_{A}(x^{-A}),
\end{equation*} 
where $g_1^{(j)}(v,\cdot) \ll_{A} T^j (1+|v|)^{-A}$.
We then change variables $v = \arcsinh(u)$, so $dv = (1+u^2)^{-1/2} du$, giving
\begin{equation}
\label{eq:JminusIntegral}
J^{-}(x, \cdot) = T^2 \intR \cos(x u) g_0(u, \cdot) du + O_{A}(x^{-A}), \qquad g_0(u, \cdot) = g_1(\arcsinh(u), \cdot) (1+u^2)^{-1/2}.
\end{equation}
Since $\arcsinh(u)$ is smooth with bounded derivatives for $u \ll 1$, then $g_0(u, \cdot)$ is Schwartz-class and satisfies $g_0^{(j)}(u, \cdot) \ll_{j,A} T^{j} (1+|u|)^{-A}$.  The integral in \eqref{eq:JminusIntegral} is a cosine transform of $g_0$, and is hence $O(T^2 (T/x)^j)$, for any $j\geq 0$, which is $O_{A}(x^{-A})$ for any $A>0$, since $x \gg T^{1+\varepsilon}$ by assumption.
\end{proof}

\section{Archimedean aspects, part 2}
\label{section:WeightFunctions2}
The goal in this section is to understand the behavior of $K^{\pm}$ defined by \eqref{eq:Kdef}.  

We begin with some comments to help bridge the material in \cite[\S 10.4]{PetrowYoung} to here.  In that article, the analog of $K$ was defined but with $J^{\pm}(x, \cdot)$ replaced by $J_{\kappa -1}(x)$, the $J$-Bessel function, with $\kappa$  fixed.  Nevertheless, a great majority of the work done on $K$ in \cite{PetrowYoung} carries over to here, and the properties of $J^{\pm}$ developed in Section \ref{section:WeightFunctions1}  will allow for this extension.

Throughout this section we assume that for some $0 <\eta \leq 1/13$ that
\begin{equation}
\label{eq:Tqsize}
 T \ll q^{\eta}.
\end{equation}
The precise $T$-dependence is not important for the proof of Theorem \ref{thm:mainthmMaassEisenstein}.

\subsection{The properties of $K$}
\begin{mylemma}[Oscillatory Case]
 \label{lemma:Kbound}
Suppose that $|m_j| \asymp M_j$ for $j=1,2,3$, and $c \asymp C$.  
 Suppose that there exists $\delta>0$ such that
 \begin{equation}
 \label{eq:OscCondition}
  \frac{\sqrt{N_1 N_2 N_3}}{C} \gg T^2 q^{\delta},
  \end{equation}
Then 
\begin{multline}
\label{eq:Kseparated2}
K^+(m_1, m_2, m_3, c) = \frac{T^2 C^{2} (N_1 N_2 N_3)^{1/2} e_c(-m_1 m_2 m_3)}{M_1 M_2 M_3}  L(m_1, m_2, m_3, c)
\\
+ O_{\delta, A}(q^{-A} \prod_{j=1}^{3} (1+ m_j)^{-2} ),
\end{multline}
where $L$ has the following properties.  Firstly, $L$ vanishes (meaning $K^+$ is very small) unless
\begin{equation}
\label{eq:MsizeOscillatoryRange}
M_j \asymp \frac{( N_1 N_2 N_3)^{1/2}}{N_j}, \quad j=1,2,3,
\end{equation}
and all the $m_j$ have the same sign.
Moreover, we have that
\begin{multline}
\label{eq:LIntegralFormulaOscillatoryRange}
L(m_1, m_2, m_3, c) =  \int_{|{\bf u}| \ll q^{\varepsilon}} \int_{|y| \ll q^{\varepsilon}} F({\bf u};y) \Big(\frac{|m_1 m_2 m_3|}{c }\Big)^{iy} 
\\
\Big(\frac{M_1}{|m_1|}\Big)^{u_1} \Big(\frac{M_2}{|m_2|}\Big)^{u_2} \Big(\frac{M_3}{|m_3|}\Big)^{u_3} \Big(\frac{C}{c}\Big)^{u_4} d{\bf u} dy,
\end{multline}
where $F = F_{ C, N_1, N_2, N_3, M_1, M_2, M_3}$ is entire in terms of ${\bf u}$, and satisfies $F({\bf u};y) \ll_{\text{Re}(\bf{u}), A} (1+|{\bf u}|)^{-A} (1+|y|)^{-A}$ for any $A>0$.

Finally, $K^-(m_1, m_2, m_3, c) \ll q^{-1000}$.
\end{mylemma}

\begin{proof}[Sketch of proof]
 The above concerns the case where $J^{+}$ is given by Lemma \ref{lemma:JplusAsymptoticExpansion}, wherein $J^+(x)$ is essentially of the form $e^{\pm ix}/\sqrt{x}$, times a smooth function of $x$.  This is the same shape of $J_{\kappa - 1}(x)$ that was used in \cite[Lem.\ 11]{PetrowYoung} and so the method used there carries over with minimal changes.
 
The final statement on the small size of $K^-$ follows immediately from Lemma \ref{lemma:JminusBound}.
\end{proof}

\begin{mylemma}[Non-oscillatory case]
\label{lemma:KboundNonOscillatory}
Suppose that $m_j \asymp M_j$ for $j=1,2,3$,  $c \asymp C$, and
\begin{equation}
\label{eq:NonOscCondition}
 \frac{\sqrt{ N_1 N_2 N_3}}{C} \ll T^2 q^{\delta}.
\end{equation}
Then for both cases $K = K^{\pm}$, we have 
\begin{multline}
\label{eq:KintegralformulaNonOscillatory}
 K(m_1, m_2, m_3, c) = T 
N_1 N_2 N_3 \Big(\frac{\sqrt{N_1 N_2 N_3}}{C}\Big)^{ } e_c(-m_1 m_2 m_3 )   
\int_{|{\bf u}| \ll T^2 q^{\delta+\varepsilon}} F({\bf u})
\\
\int_{|t| \ll q^{\varepsilon} + P} f(t) \Big(\frac{|m_1 m_2 m_3|  }{c }\Big)^{it}  
\Big(\frac{M_1}{|m_1|}\Big)^{u_1} \Big(\frac{M_2}{|m_2|}\Big)^{u_2} \Big(\frac{M_3}{|m_3|}\Big)^{u_3} \Big(\frac{C}{c}\Big)^{u_4} dt d{\bf u}
 \\
  + O_{\eps, A}(q^{-A} \prod_{j=1}^{3} (1+ m_j)^{-2} ),
\end{multline}
where $P$ is defined by 
\begin{equation}
\label{eq:Pdef}
P = \frac{M_1 M_2 M_3}{C},
\end{equation}
$f(t) \ll (1+|t|)^{-1/2}$, and $F({\bf u}) \ll_{A, \text{Re}({\bf u}), \eps} q^{\varepsilon} \prod_{\ell=1}^{4} (1+\frac{|u_{\ell}|}{q^{\varepsilon}})^{-A}$ for all $A>0$.  Moreover, $F$ vanishes (meaning $K$ is small) unless
\begin{equation}
\label{eq:smallkiMi}
 \frac{M_1 N_1}{C} \ll_{\eps} q^{\varepsilon}, \qquad  \frac{M_2 N_2}{C} \ll_{\eps} q^{\varepsilon}, \qquad  \frac{M_3 N_3}{C} \ll_{\eps} q^{\varepsilon}.
\end{equation}
If there exists $\eps>0$ such that $P \gg q^{\varepsilon}$ then $f$ may be chosen to have support on $|t| \asymp P$.
\end{mylemma}
\begin{proof}[Sketch of proof]
 In this case, $J^{\pm}$ satisfies the properties of Lemma \ref{lemma:Jpluspropertiesxsmall} or \ref{lemma:Jminuspropertiesxsmall} (depending on the choice of $\pm$).  In turn, these are essentially the only properties that were used about $J_{\kappa-1}(x)$ in \cite[Lem.\ 12]{PetrowYoung}.
 \end{proof}

\begin{mylemma}[Other cases]
\label{lemma:KboundOther}
Suppose some $m_j = 0$, and let $K$ denote either case of $K^{\pm}$.  If \eqref{eq:OscCondition} holds, then $K$ is small.  If \eqref{eq:NonOscCondition} holds, then $K$ is small unless $|m_j| \ll_{\eps} \frac{C}{N_j} q^{\varepsilon}$ for $j=1,2,3$, in which case
\begin{equation}
\label{eq:KboundOther}
K(m_1, m_2, m_3;c) \ll_{\eps} T N_1 N_2 N_3 \Big(\frac{\sqrt{N_1 N_2 N_3}}{C}\Big)^{ }   q^{\varepsilon}. 
\end{equation}
\end{mylemma}
\begin{proof}
The fact that $K$ is small if \eqref{eq:OscCondition} holds follows from repeated integration by parts (see \cite[Lem.\ 8.1]{BKY} for instance).  If \eqref{eq:NonOscCondition} holds, then another repeated integration by parts argument shows that the integral is small if there exists $\eps>0$ such that $|m_j| \gg \frac{C}{N_j} q^{\varepsilon}$ for some $j$.  Finally, the bound \eqref{eq:KboundOther} follows from trivially estimating the integral defining $K$, using \eqref{eq:Jplusbound} or \eqref{eq:Jminusbound}.
\end{proof}

\section{Completing the proof of Theorem \ref{thm:mainthmMaassEisenstein}}
\label{section:Finale}
Here we finish the proof of the bounds $\mathcal{T}^\pm, \mathcal{T}_0^\pm \ll_{\eps} T^B q^{\eps}$ (for definitions, see \eqref{eq:Tpmdef} and \eqref{eq:T0pmdef}), which will complete the proof of Theorem \ref{thm:mainthmMaassEisenstein}.

We only deal with the case that $\epsilon_j=1$ for all $j=1,2,3$. The other cases are similar. Recall the definition of $\mathcal{T}^\pm$ from \eqref{eq:Tpmdef}: 
\begin{equation*}
\mathcal{T}^\pm:= \frac{1}{C  \sqrt{N_1 N_2 N_3}} \sum_{\substack{m_1, m_2, m_3, r \geq 1 \\ (m_1, r) = 1}} G(m_1,m_2,m_3 ; qr) K^{\pm}(m_1,m_2,m_3, qr) .
\end{equation*}
Using \eqref{eq:GHchirelation}, we have
\begin{equation*}
 |\mathcal{T}^\pm| \ll \frac{1}{C^{2} q \sqrt{N_1 N_2 N_3}} 
\Big| \sum_{\substack{m_1, m_2, m_3, r \geq 1 \\ (m_1, r) = 1}} e_{qr}(m_1 m_2 m_3) K^{\pm}(m_1, m_2, m_3, qr) H_{\chi}(\pm m_1, m_2, m_3, r) \Big|.
\end{equation*}
Letting $N= N_1 N_2 N_3$, 
the behavior of $K$ depends on whether or not
\begin{equation}
\label{eq:oscillatorycondition}
 \frac{\sqrt{N}}{C} \gg T^2 q^{\varepsilon}.
\end{equation}

\textbf{Oscillatory case.}  Suppose \eqref{eq:oscillatorycondition} holds for some $\eps>0$.  By Lemma \ref{lemma:Kbound}, only the case of $K^{+}$ is relevant, in which case we have (recalling \eqref{eq:Zdef})
\begin{multline}
\label{eq:TfancyBound1}
 |\mathcal{T}^{+}| 
 \ll \frac{T^2}{ q  M} 
 \Big| \int_{|{\bf u}| \ll q^{\varepsilon}} \int_{|y| \ll q^{\varepsilon}} F({\bf u};y) q^{-iy}
 M_1^{u_1} M_2^{u_2} M_3^{u_3} (C/q)^{u_4}
\\
Z(u_1-iy, u_2-iy, u_3-iy, u_4+iy) d{\bf u} dy \Big|,
\end{multline}
plus a small error term, where $M = M_1 M_2 M_3$.  Here we initially take $\text{Re}(u_j) = 1+\varepsilon$ for all $j$.
According to Lemma \ref{lemma:Zproperties}, 
write $Z= Z_0 + Z_1$.  For $Z_0$, we keep the lines at $1+\varepsilon$, while for $Z_1$ we move them to $1/2+\varepsilon$.  By the decay properties of $F$, the horizontal contour integrals arising from these contour shifts are small ($\ll q^{-100}$, say), and we will not mention them further. 
Thus we obtain
\begin{equation}
\label{eq:fancyTbound}
 \mathcal{T}^{+} \ll \frac{q^{\varepsilon} T^2}{q M} \Big(\frac{MC}{q} + \frac{\sqrt{MC}}{\sqrt{q}} q^{3/2}\Big)
 \ll q^{\varepsilon} T^2 \Big(\frac{C}{q^2} + \frac{\sqrt{C}}{N^{1/4}}\Big),
\end{equation}
using that $K^{+}$ is very small unless $M \asymp \sqrt{N}$ in this oscillatory case.
Since $C T^2 \ll N^{1/2}  \ll (qT)^{3/2+\varepsilon}$ (from \eqref{eq:oscillatorycondition} and \eqref{eq:NvariableSizes}), we have $\mathcal{T}^\pm \ll T  q^{\varepsilon}$ (using $T \ll q^{\eta}$ for some $\eta > 0$ small).

\textbf{Non-oscillatory case.}  
The method of estimation is similar in the case that $\frac{\sqrt{N}}{C} \ll T^2 q^{\varepsilon}$, but we use  Lemma \ref{lemma:KboundNonOscillatory} in place of Lemma \ref{lemma:Kbound}.  From the terms with $m_j \asymp M_j$, we obtain that the contribution to $\mathcal{T}^\pm$ is
\begin{equation*}
  \ll \frac{N T }{C^3  q} \Big| \int_{|t| \ll q^{\varepsilon} +P} f(t)
 \int_{|{\bf u}| \ll T^2 q^{\varepsilon}} M_1^{u_1} M_2^{u_2} M_3^{u_3} \Big(\frac{C}{q}\Big)^{u_4}
 \frac{F({\bf u})}{q^{it}} 
 Z(u_1-it, u_2-it, u_3-it, u_4+it) d{\bf u} dt \Big|,
\end{equation*}
where $P= M/C$. By the large sieve-like bound \eqref{eq:LSlike}, we have that the contribution to the above from $Z_1$, say $\mathcal{T}^\pm_1$, satisfies the bound
\begin{equation*}
 \mathcal{T}^\pm_1 \ll_{\eps} \frac{N T  q^{\varepsilon}}{C^3 q}  \frac{\sqrt{M C}}{\sqrt{q}} q^{3/2} \Big(1+\frac{\sqrt{M }}{\sqrt{C}} \Big)
 T^2
 .
\end{equation*}
In this case, $M  \ll_{\eps} \frac{C^3}{N} q^{\varepsilon}$, and so this bound becomes
\begin{equation}
\label{eq:T1boundNonOscillatory}
\mathcal{T}^\pm_1 \ll_{\eps} q^{\varepsilon} T^{3} (\frac{\sqrt{N}}{C} + 1) \ll T^{5} q^{\varepsilon}.
\end{equation}

Next consider the contribution from $Z_0$, say $\mathcal{T}^\pm_{00}$.  If $P \gg q^{\varepsilon}$ for some $\eps>0$, 
then we may assume $f$ is supported on $|t| \asymp P$, and
we shift the contours to the $(1/2+\varepsilon)$-line.  No poles are crossed during this procedure since they occur at height $t$, and the horizontal integrals arising from this contour shift are negligible since $F$ is small at this height.
By the final sentence of Lemma \ref{lemma:Zproperties}, the bound we obtain on $\mathcal{T}^\pm_{00}$ is no worse than the bound on $\mathcal{T}^\pm_1$ given in \eqref{eq:T1boundNonOscillatory}.  

Finally, consider the case $P \ll_{\eps} q^{\varepsilon}$, that is, $M  \ll_{\eps} C q^{\varepsilon}$.  Here we keep the contours at the $(1+\varepsilon)$-line, giving
\begin{equation*}
 \mathcal{T}^\pm_{00} \ll_{\eps} \frac{N T }{C^3 q} \frac{M  C}{q} T^2 q^{\varepsilon} 
 \ll \frac{N T^{3} }{C q^2} q^{\varepsilon} =  T^{3}  \frac{\sqrt{N}}{C} \frac{\sqrt{N}}{q^2} q^{\varepsilon} \ll T^{\frac{13}{2}} q^{-\frac12+\varepsilon},
\end{equation*}
using \eqref{eq:NvariableSizes},
which is $\ll_{\eps}   q^{\varepsilon}$ taking $\eta \leq 1/13$ in \eqref{eq:Tqsize}.

\textbf{The cases with some $m_j = 0$.}
We will estimate $\mathcal{T}^\pm_0$ by trivial bounds.
By Lemma \ref{lemma:KboundOther}, $K(m_1, m_2, m_3, c)$ is very small in this case, unless we are in the non-oscillatory situation \eqref{eq:NonOscCondition}.  

Using 
Lemma \ref{lemma:KboundOther}, we deduce 
\begin{equation}
\label{eq:SfromSomeMjZero}
\mathcal{T}_0^\pm\ll_{\eps}  \frac{TN }{C^2} q^{\varepsilon} \sum_{r \asymp C/q} \sum_{\substack{m_1 m_2 m_3 =0 \\ |m_j| \ll_{\eps} M_j}}  |G(m_1, m_2, m_3 ; qr)|, \qquad M_j := \frac{C}{N_j} q^{\varepsilon},
\end{equation}
plus a small error term.  
Recall the bound \eqref{eq:GtrivialBound},
and that $G(m_1, m_2, m_3, qr) = 0$ if $(m_1, r) \neq 1$.

First consider the terms with $m_3 = 0$ and $m_1, m_2 \neq 0$.  Their contribution to \eqref{eq:SfromSomeMjZero} is
\begin{equation}
\label{eq:BoundFromm3=0}
  \ll_\eps  \frac{TN}{C^2} \frac{C}{q} M_1 M_2  \frac{q^{\varepsilon}}{C q} 
  \ll T  \frac{N_3}{q^2} q^{\varepsilon} \ll T^3 q^{\varepsilon},
\end{equation}
using \eqref{eq:NvariableSizes}.
The case with $m_2 = 0$ and $m_1, m_3 \neq 0$ is essentially identical to the previous case, but
the case with $m_1 = 0$ and $m_2, m_3 \neq 0$ is slightly different because of the condition $(m_1, r) = 1$.  The $r$-sum collapses to $r=1$, and this sum is even smaller than that appearing in the previous cases (essentially, the factor $\frac{C}{q}$ may be improved to $1$).

Next consider the terms with two $m_j=0$, the hardest one being $m_2 = m_3 = 0$.  Compared to \eqref{eq:BoundFromm3=0}, the difference is that the factor $M_2$ is replaced by $q$, leading to the bound
\begin{equation*}
\ll_\eps  \frac{TN}{ C^2} \frac{C}{q} M_1 q  \frac{q^{\varepsilon}}{C q} 
  \ll T  \frac{N_2 N_3}{q  C} q^{\varepsilon} \ll T  \frac{\sqrt{N}}{C} \frac{\sqrt{N_2 N_3}}{q} q^{\varepsilon}.
\end{equation*}
Using $\frac{\sqrt{N}}{C} \ll_{\eps} T^2 q^{\varepsilon}$ and $N_2 N_3 \ll_{\eps} (qT)^{2+\varepsilon}$ (recall \eqref{eq:NvariableSizes})  shows this is $\ll_{\eps} T^4 q^{\varepsilon}$.  If $m_1$ is one of the two $m_j$'s equal to zero, then the numerology changes enough to be worthy of mention (we no longer have $N_1 N_3 \ll_{\eps} (Tq)^{2+\varepsilon}$, but on the other hand the $r$-sum collapses, so we may assume $C \asymp q$ since $c = qr \asymp C$).  Say $m_1 = m_3 = 0$ and $m_2 \neq 0$.  Then the contribution of these terms to $\mathcal{T}_0^\pm$ is
\begin{equation*}
\ll_{\eps}   \frac{TN}{ C^2} q M_2  \frac{q^{\varepsilon}}{C q} 
\ll
   \frac{TN}{ C^3} M_2   q^{\varepsilon}
 \ll
    \frac{T N_1 N_3}{C^2} q^{\varepsilon} \ll T^5 q^{\varepsilon},
\end{equation*}
where we used $\frac{C}{q} \asymp 1$, $M_2 \ll \frac{C}{N_2} q^{\varepsilon}$, and $N_1 N_3 \leq N \ll C^2 T^4 q^{\varepsilon}$.  

Finally, the terms of $\mathcal{T}_0^\pm$ with $m_1 = m_2 = m_3 = 0$ (hence $r=1$, $C \asymp q$) are bounded by
\begin{equation*}
\ll_{\eps}  \frac{TN}{C^2}  \frac{q}{C} q^{\varepsilon} 
 \ll  T^5 q^{\varepsilon}.
\end{equation*}

This completes the proof of Theorem \ref{thm:mainthmMaassEisenstein}.

\section{Sketch of proof of Theorems \ref{thm:mainthmHybridVersion} and \ref{thm:mainthmHolomorphic}}
\label{section:Epilogue}
In this section, we outline what changes are needed to prove Theorem \ref{thm:mainthmHybridVersion}.  
The problem is arithmetically identical to the proof of Theorem \ref{thm:mainthmMaassEisenstein}, but the Archimedean aspects are different.  Recall we have assumed that $T \gg q^{\eta}$ for some small but fixed $\eta > 0$.

The first change is that instead of using $h_0(t)$ defined by \eqref{eq:h0def}, we take
\begin{equation*}
 h_0(t) = \frac{1}{\cosh\big(\frac{t-T}{\Delta}\big)} + \frac{1}{\cosh\big(\frac{t+T}{\Delta}\big)},
\end{equation*}
as in \cite[\S 4]{YoungCubic}, where $\Delta = T^{\varepsilon}$ for some $\eps>0$.  A more precise version of Lemma \ref{lemma:Vproperties} is developed in \cite[\S 5]{YoungCubic}, showing that $V_j(y,t)$ has an asymptotic expansion with leading term of the form $W_j(\frac{y}{T^j})$, where $W_1$ and $W_2$ are fixed smooth weight functions, satisfying $x^k W_j^{(k)}(x) \ll_{A} (1+|x|)^{-A}$ for all $A>0$.
The analogs of the estimates for $J^{\pm}$ appear as \cite[Lem.\ 7.1, 7.2]{YoungCubic}, while the crucial integral representations of $K(m_1, m_2, m_3, c)$ are treated in \cite[Lem.\ 8.1]{YoungCubic} in place of those covered in Section \ref{section:WeightFunctions2}.  Note that in \cite[(8.5)]{YoungCubic}, the contours were set at $\text{Re}({\bf y}) = \text{Re}(u) = 0$.  To accommodate more general choices of contour, the formula \cite[(8.4)]{YoungCubic} should be updated to state
\begin{equation*}
 K^{+}(m_1, m_2, m_3, c) = 
 \frac{ C^{3/2} \Delta T (N_1 N_2 N_3)^{1/2} e_c(-m_1 m_2 m_3)}{(M_1 M_2 M_3)^{1/2}}  L(m_1, m_2, m_3, c),
\end{equation*}
plus a small error term, where
\begin{multline*}
L(m_1, m_2, m_3, c) =  \frac{1}{V} \int_{|{\bf u}| \ll (qT)^{\varepsilon}} \int_{|y| \ll U} F({\bf u};y) \Big(\frac{|m_1 m_2 m_3|}{c }\Big)^{iy} 
\\
\Big(\frac{M_1}{|m_1|}\Big)^{u_1} \Big(\frac{M_2}{|m_2|}\Big)^{u_2} \Big(\frac{M_3}{|m_3|}\Big)^{u_3} \Big(\frac{C}{c}\Big)^{u_4} d{\bf u} dy,
\end{multline*}
where $V=T$ and 
\begin{equation*}
 U = \frac{T^2 C}{(N_1 N_2 N_3)^{1/2}}.
\end{equation*}
Moreover, $L$ vanishes (i.e., $K^+$ is very small) unless
\begin{equation*}
 C \ll_{\eps} \frac{(N_1 N_2 N_3)^{1/2}}{\Delta^{1-\varepsilon} T} 
 \qquad
 \text{and}
 \qquad
 M_j \asymp \frac{(N_1 N_2 N_3)^{1/2}}{N_j}, j=1,2,3.
\end{equation*}
The formula for $K^{-}$ can be adapted in a similar way, but we leave out the details for brevity.

Now if we follow along the details of the \textbf{Oscillatory case} from Section \ref{section:Finale}, we obtain that the contribution to $\mathcal{T}^\pm$ from these terms is (in place of \eqref{eq:TfancyBound1})  
\begin{multline*}
 |\mathcal{T}^\pm| 
 \ll \frac{\Delta T}{q C^{1/2} M^{1/2} V} 
 \Big| \int_{|{\bf u}| \ll q^{\varepsilon}} \int_{|y| \ll U } F({\bf u};y) q^{-iy}
 M_1^{u_1} M_2^{u_2} M_3^{u_3} (C/q)^{u_4}
\\
Z(u_1-iy, u_2-iy, u_3-iy, u_4+iy) d{\bf u} dy \Big|,
\end{multline*}
plus a small error term.  We decompose $Z$ as $Z_0 + Z_1$, and for $Z_1$ we shift the contour to the $(1/2+\varepsilon)$-lines, giving that its contribution to $\mathcal{T}^\pm$ is
\begin{equation*}
  \ll_{\eps} 
 \frac{1}{ q} 
\frac{  \Delta T}{C^{1/2} M^{1/2}}
\frac{U}{V} 
 \frac{\sqrt{MC}}{\sqrt{q}} q^{3/2}
 T^{\varepsilon}.
\end{equation*}
Using $\frac{U}{V} \ll_{\eps} \Delta^{-1+\varepsilon}$ shows this term is $\ll_{\eps} T^{1+\varepsilon}$, which is the bound required for Theorem \ref{thm:mainthmHybridVersion}.  Next we turn to $Z_0$.  For this term, it is helpful to point out that in fact $F({\bf u};y)$ is very small unless $|y| \asymp U$, which was a property that was not stated in  \cite[Lem.\ 8.1]{YoungCubic}, but was developed in the proof (see \cite[p.1569]{YoungCubic}).  This shows that if $U \gg T^{\varepsilon}$ for some $\eps>0$, then in the estimation of $Z_0$ we can shift the contours to the $(1/2+\varepsilon)$-lines without crossing poles.  The bound obtained on $Z_0$ is no larger than the one obtained on $Z_1$.  If $U \ll_{\eps} T^{\varepsilon}$, then we keep the contours at the $(1+\varepsilon)$-lines, giving that their contribution to $\mathcal{T}^\pm$ is
\begin{equation*}
 \ll_{\eps} \frac{\Delta T}{q C^{1/2} M^{1/2} V} \frac{MC}{q} T^{\varepsilon} \ll q^{-1/2} \Delta^{1/2} T^{1+\varepsilon},
\end{equation*}
which is stronger than the bound obtained on $Z_1$.

The \textbf{Non-oscillatory case} is similar, and we omit the details for brevity.

Finally, we need to consider the terms where some $m_j=0$.  These cases were overlooked in \cite{YoungCubic}, so we take this opportunity to correct this omission.  The first claim is that $K^+(m_1, m_2, m_3, c)$ is very small if some $m_j = 0$.  This follows from the fact that $B^+(x)$ (the analog of $J^+(x,\cdot)$) is very small unless $x \gg_{\eps} \Delta T^{1-\varepsilon}$, in which case it has an asymptotic expansion of the form $\frac{\Delta T}{\sqrt{x}} \cos(x+\phi(x,T))$, where $\phi(x,T) = - 2T^2/x + \dots$.   Then repeated integration by parts in the $t_j$ variable (where $m_j = 0$) shows that $K^+$ is small.  Therefore, it suffices to consider $K^-$.  We claim that if some $m_j = 0$ then
\begin{equation}
\label{eq:KminusBoundSomemjZero}
 K^{-}(m_1, m_2, m_3, c) \ll_{\eps}  \Delta N T^{\varepsilon}.
\end{equation}
The trivial bound arising from \cite[Lem.\ 7.2]{YoungCubic} would give a bound of the form $NT$, so \eqref{eq:KminusBoundSomemjZero} saves a factor of $T/\Delta$ over this.  We now prove the claim.  According to \cite[(7.3)]{YoungCubic}, we have
\begin{equation*}
 B^-(x) = \Delta T \int_{|v| \leq \Delta^{-1+\varepsilon}}
 \cos(x \sinh v) e^{2ivT} g(\Delta v) dv + O_{A}(T^{-A}),
\end{equation*}
where $g^{(j)}(x) \ll_{A} (1+|x|)^{-A}$ for all $A>0$.  Moreover, $B^-(x)$ is very small unless $x \asymp T$.  Here $B^{-}(x)$ is the analog of $J^{-}(x,\cdot)$.  To fix the notation, say $m_3 = 0$ (the cases with $m_1 = 0$ or $m_2 = 0$ are identical).  Then the $t_3$-integral inside the definition of $K^{-}$ takes the form
\begin{equation*}
 \intR w(t_3, \cdot) \cos\Big(\frac{4 \pi \sqrt{t_1 t_2 t_3}}{c} \sinh v \Big) dt_3,
\end{equation*}
where $w(t_3, \cdot)$ is supported on $t_3 \asymp N_3$, and satisfies $t_3^j \frac{d^j}{dt_3^j} w(t_3, \cdot) \ll 1$.  Repeated integration by parts (see \cite[Lem.\ 8.1]{BKY}) therefore shows that $K^-(m_1, m_2, m_3, c)$ is very small unless 
\begin{equation*}
 \frac{\sqrt{N}}{C} |v| \ll_{\eps} T^{\varepsilon}.
\end{equation*}
On the other hand, we also know $K^-$ is very small unless $x \asymp \frac{\sqrt{N}}{C} \asymp T$, so inside the definition of $K^-$ we may further restrict $v$ by $|v| \ll_{\eps} T^{-1+\varepsilon}$.  The trivial bound on $K^{-}$ now leads to 
\eqref{eq:KminusBoundSomemjZero}.  An integration by parts argument in the $t_1, t_2$ variables shows that $K^{-}(m_1, m_2, 0, c)$ is very small unless $|m_j| \ll_{\eps} \frac{C}{N_j} T^{\varepsilon}$, for $j=1,2$.  

At this point, we carry through the same argument used in Section \ref{section:Finale}, using \eqref{eq:GtrivialBound} as before, but using  \eqref{eq:KminusBoundSomemjZero} in place of Lemma \ref{lemma:KboundOther}.  As a representative sample, consider the contribution from $m_3 = 0$, $m_1, m_2 \neq 0$.  These terms give
\begin{equation*}
 \ll_{\eps} \frac{1}{C \sqrt{N}} \frac{\Delta N}{Cq} T^{\varepsilon} 
 \sum_{r \asymp \frac{C}{q}} \sum_{1 \leq |m_1| \ll_{\eps} \frac{C}{N_1} T^{\varepsilon}}
 \sum_{1 \leq |m_2| \ll_{\eps} \frac{C}{N_2} T^{\varepsilon}}
 (m_2, q)(m_3, q) 
\ll \frac{\Delta N_3}{q^2 T} T^{\varepsilon},
 \end{equation*}
using $C \asymp \frac{\sqrt{N}}{T}$.  Since $N_3 \ll_{\eps} (q T)^{2+\varepsilon}$, this is $\ll_{\eps} \Delta T^{1+\varepsilon}$, which is the bound required for Theorem \ref{thm:mainthmHybridVersion}.  Similar arguments may be used to treat the other terms with $m_1 m_2 m_3 = 0$, and we leave the details to the diligent reader.

The proof of Theorem \ref{thm:mainthmHybridVersion} is now complete.

Finally, we discuss the proof of Theorem \ref{thm:mainthmHolomorphic}.  The framework of \cite{YoungCubic} placed both the Maass forms and holomorphic forms on an equal footing, and so the proof of the hybrid bound \eqref{eq:holomorphicFormsBoundkappaHybridVersion} is now essentially identical to that of Theorem \ref{thm:mainthmHybridVersion}.  In order to derive the bound \eqref{eq:holomorphicFormsBoundkappaSmallVersion}, one may adapt the material from Section \ref{section:WeightFunctions1}.  It is not difficult to prove an analogous version of Lemma \ref{lemma:Vproperties} (the use of Stirling's formula is slightly different).  The use of the Bruggeman-Kuznetsov formula will then be replaced by the Petersson formula and Poisson summation over $\kappa$ (see \cite[p.85-86]{IwaniecClassical}).  One can then derive properties of the resulting weight functions which are analogous to those of $J^{\pm}$ presented in Sections \ref{section:Jplus} and \ref{section:Jminus}.  The properties of $K^{\pm}$ derived in Section \ref{section:WeightFunctions2}  then carry over with minimal changes, and the final steps of Section \ref{section:Finale} then proceed in the same fashion as in the proof of Theorem \ref{thm:mainthmMaassEisenstein}.


\begin{thebibliography}{99}

  \bibitem[ALe]{AtkinLehner} A. O. L. Atkin, and J. Lehner, 
\emph{Hecke operators on {$\Gamma _{0}(m)$}.}
Math. Ann. 185 1970 134--160. 
  \bibitem[ALi]{AL78} A. O. L. Atkin and W. Li, \emph{Twists of Newforms and Pseudo-Eigenvalues of $W$-Operators.} Invent.~Math. 48  (1978), 221--243. 
 

\bibitem[BLT]{BarbanLinnikTshudakov} M. B. Barban, Yu. V. Linnik, and N. G. Tshudakov, 
\emph{On prime numbers in an arithmetic progression with a prime-power difference.}
Acta Arith. 9 1964 375--390. 



\bibitem[BH]{BlomerHarcos} V. Blomer and G. Harcos, \emph{
Hybrid bounds for twisted $L$-functions.}
J. Reine Angew. Math. 621 (2008), 53--79. 
\newblock Addendum: {H}ybrid bounds for twisted {$L$}-functions.
\newblock {\em J. Reine Angew. Math.}, 694 (2014), 241--244. 

 
\bibitem[BHM]{BlomerHarcosMichel} V. Blomer, G. Harcos, and P. Michel, 
\emph{A Burgess-like subconvex bound for twisted $L$-functions. }
Appendix 2 by Z. Mao.
Forum Math. 19 (2007), no. 1, 61--105. 
 
  \bibitem[BKY]{BKY} V. Blomer, R. Khan, and M. Young, \emph{Distribution of mass of holomorphic cusp forms.} 
 Duke Math. J. 162 (2013), no. 14, 2609--2644.

\bibitem[BM]{BlomerMili} V. Blomer and D. Mili{\'c}evi{\'c},
\emph{{$p$}-adic analytic twists and strong subconvexity}. 
Ann. Sci. \'Ec. Norm. Sup\'er. (4)  48 (2015), no. 3, 561--605. 

\bibitem[B]{Burgess} D. A. Burgess, \emph{On character sums and {$L$}-series. {II}}. Proc. Lond. Math. Soc. (3) 13 (1963), 524--536.

  
  \bibitem[CI]{CI}  J. B. Conrey and H. Iwaniec, \emph{The cubic moment of central values of automorphic $L$-functions.} Ann. of Math. (2) 151 (2000), no. 3, 1175--1216.



\bibitem[D1]{SGA45} P. Deligne, \emph{Cohomologie \'{e}tale}, S\'{e}minaire de g\'{e}om\'{e}trie alg\'{e}brique du Bois-Marie SGA
              $4\frac{1}{2}$, Lecture Notes in Mathematics \textbf{569}, Springer-Verlag, Berlin, 1977.


\bibitem[D2]{DeW2} P. Deligne, \emph{La conjecture de {W}eil, II.}  Inst. Hautes \'Etudes Sci. Publ. Math. no. 52 (1980), 137--252.

\bibitem[DFI]{DFI} W. Duke, J.B. Friedlander, and H. Iwaniec, \emph{
The subconvexity problem for Artin $L$-functions.}
Invent. Math. 149 (2002), no. 3, 489--577. 
 

\bibitem[FKM]{FKMcoh}
E.~Fouvry, E.~Kowalski, and P.~Michel. On the conductor of cohomological transforms. {\em Ann. Fac. Sci. Toulouse Math. (6)}, 2019.
\newblock To appear.

\bibitem[FKMS]{MichelAppliedCoh}
E.~Fouvry, E.~Kowalski, P.~Michel, and W.~Sawin. Lectures on applied {$\ell$}-adic cohomology. In {\em Analytic Methods in Arithmetic Geometry}, pages 113--195.
  Contemporary Mathematics, Vol. 740. American Mathematical Society,
  Providence, RI, 2019.


\bibitem[GR]{GR} I. S. Gradshteyn,  I. M. Ryzhik, \emph{Table of integrals,
series, and products}. Translated from the Russian. Sixth edition. Translation edited and with a preface by Alan Jeffrey and Daniel Zwillinger. Academic Press, Inc., San Diego, CA, 2000.

\bibitem[Ge]{Gelbart}
S.~S. Gelbart.
 {\em Automorphic forms on ad\`ele groups}.
Princeton University Press, Princeton, N.J.; University of Tokyo
  Press, Tokyo, 1975.
 Annals of Mathematics Studies, No. 83.

\bibitem[Gu]{Guo} J. Guo, \emph{On the positivity of the central critical values of
              automorphic {$L$}-functions for {${\rm GL}(2)$}.}
Duke Math. J. 83 (1996), no. 1, 157--190. 

\bibitem[H-B]{HeathBrownHybrid} D. R. Heath-Brown  \emph{Hybrid bounds for Dirichlet $L$-functions}. Invent. Math. 47 (1978), no. 2, 149--170.

 \bibitem[HL]{HoffsteinLockhart} J. Hoffstein and P. Lockhart, {\it Coefficients of Maass forms and the Siegel zero.} With an appendix by Dorian Goldfeld, Hoffstein and Daniel Lieman. Ann. of Math. (2) 140 (1994), no. 1, 161--181.  
\bibitem[Iv]{Ivic} A. Ivi{\'c}, \emph{On sums of Hecke series in short intervals.} J. Th\'{e}or. Nombres Bordeaux 13 (2001), no. 2, 453--468.

 \bibitem[Iw1]{IwaniecSmallEigenvalues} H. Iwaniec, {\it Small eigenvalues of Laplacian for $\Gamma_{0}(N)$.} Acta Arith. 56 (1990), no. 1, 65--82.
  \bibitem [Iw2]{IwaniecClassical} H. Iwaniec, \emph{Topics in Classical Automorphic Forms}, Grad. Stud.
        Math., vol 17, Amer. Math. Soc., 1997.
 
\bibitem[IK]{IK}  H. Iwaniec, E. Kowalski, \emph{Analytic number theory}, AMS Colloquium Publications \textbf{53}, American Mathematical Society, Providence, RI, 2004.

\bibitem[JL]{JacquetLanglands}
H. Jacquet and R. P. Langlands.
{\em Automorphic forms on {${\rm GL}(2)$}}.
 Lecture Notes in Mathematics, Vol. 114. Springer-Verlag, Berlin-New
  York, 1970.

\bibitem[Ka1]{KatzGaussSums} N. Katz, \emph{Gauss sums, {K}loosterman sums, and monodromy groups}, Annals of Mathematics Studies \textbf{116}, Princeton University Press, Princeton, NJ, 1988.

\bibitem[Ka2]{KatzBettiNos}
N. Katz, \emph{Sums of {B}etti numbers in arbitrary characteristic},
Finite Fields Appl. 7 (2001), no. 1, 29--44.

\bibitem[KPY]{KiralPetrowYoung} E. M. Kiral, I. Petrow and M. Young,
\emph{Oscillatory integrals with uniformity in parameters.} J. Th\'eor. Nombres Bordeaux, 31 (2019), no. 1, 145--159.

\bibitem[KL13]{KLkuznetsov}
A.~Knightly and C.~Li.
\emph{Kuznetsov's trace formula and the {H}ecke eigenvalues of {M}aass
  forms.} Mem. Amer. Math. Soc., 224 (2013), no. 1055.

\bibitem[KMS]{KMS}
E.~Kowalski, Ph.~Michel, and W.~Sawin,
\emph{Bilinear forms with {K}loosterman sums and applications}, Annals of
Math.\ 186 (2017), 413--500.

\bibitem[Lau]{Laumon}
G.~Laumon,
\emph{Semi-continuit\'e du conducteur de {S}wan (d'apr\`es {P}.\ {D}eligne)}, in “Caract\'eristique d'{E}uler–
{P}oincar\'e”, Ast\'erisque 83 (1981), 173--219.

\bibitem[Li]{LiEpsFactorsAndnCloseness1}
W.~C.~W. Li.
\newblock On the representations of {${\rm GL}(2)$}. {I}. {$\varepsilon
  $}-factors and {$n$}-closeness.
\newblock {\em J. Reine Angew. Math.}, 313:27--42, 1980.

\bibitem[Mil]{Milicevic} D. Mili\'cevi\'c, \emph{Sub-{W}eyl subconvexity for {D}irichlet {$L$}-functions to
              prime power moduli}.  
Compos. Math. 152 (2016), no. 4, 825--875. 


\bibitem[MV]{MontgomeryVaughan} H. Montgomery and R. Vaughan,
\emph{Multiplicative number theory. I. Classical theory.}
Cambridge Studies in Advanced Mathematics \textbf{97}. Cambridge University Press, Cambridge, 2007.

 \bibitem[P1]{Petrow} I. Petrow, \emph{A twisted Motohashi formula and Weyl-subconvexity for $L$-functions of weight two cusp forms.} Math. Ann., 363 (2015), no. 1-2, 175--216.

\bibitem[P2]{PetrowPetersson} I. Petrow, \emph{Bounds for traces of Hecke operators and applications to modular and elliptic curves over a finite field.} Algebra Number Theory, 12 (2018), no. 10, 2471--2498.

\bibitem[PY1]{PetrowYoung} I. Petrow and M. Young, \emph{A generalized cubic moment and the Petersson formula for newforms}.  Math. Ann., 373 (2019), no. 1-2, 287--353.

\bibitem[PY2]{PetrowYoungFourth} I. Petrow and M. Young, \emph{The fourth moment of Dirichlet $L$-functions along a coset and the Weyl bound}, arXiv:1908.10346 (2019). 

\bibitem[Wa]{Waldspurger} J.L. Waldspurger,
\emph{Sur les valeurs de certaines fonctions {$L$} automorphes en
              leur centre de sym\'etrie}. Compos. Math. 54 (1985), no. 2, 173--242.

\bibitem[Wu1]{WuGL1timesGL2}  H. Wu, \emph{Burgess-like subconvex bounds for
              {$\text{GL}_2\times\text{GL}_1$}.} Geom. Funct. Anal. 24 (2014), no. 3, 968--1036.

\bibitem[Wu2]{WuGL1} H.~Wu.
\emph{Burgess-like subconvexity for {${\rm GL}_1$}.}
Compos. Math. 155 (2019), no. 8, 1457--1499.
              
 \bibitem[Y1]{YoungCubic} M. Young, \emph{Weyl-type hybrid subconvexity bounds for twisted $L$-functions and Heegner points on shrinking sets},  J. Eur. Math. Soc. (JEMS) 19 (2017), no. 5, 1545--1576.
 
\bibitem[Y2]{YoungEisenstein} M. Young, \emph{Explicit calculations with Eisenstein series}, J. Number Theory 199 (2019), 1--48.

 \end{thebibliography}
\end{document}